\title{Icosahedral invariants and a construction of class fields via periods of $K3$ surfaces}
\author{Atsuhira Nagano}
\documentclass[10pt]{article}
\usepackage{mathrsfs,bm,graphics,amsfonts,amsthm,amssymb,amsmath,amscd,booktabs}
\usepackage[dvips]{graphicx,color}
\def\bigzerou{\smash{\lower1.7ex\hbox{\b 0}}}

\newtheorem{thm}{Theorem}[section]
\newtheorem{df}{Definition}[section]
\newtheorem{lem}{Lemma}[section]
\newtheorem{prop}{Proposition}[section]
\newtheorem{rem}{Remark}[section]
\newtheorem{exap}{Example}[section]

\newtheorem{cor}{Corollary}[section]

\def\comment#1{{ }}
\makeatletter
  
      \@addtoreset{equation}{section}
      
\begin{document}
\maketitle

\begin{abstract}
In the theory of complex multiplication, it is important to construct class fields over CM fields.
In this paper, we consider explicit $K3$ surfaces parametrized by Klein's icosahedral invariants.
Via the periods and the Shioda-Inose structures of $K3$ surfaces,  the special values of icosahedral invariants generate class fields over quartic CM fields.
 Moreover, we give an explicit expression of the canonical model of the Shimura variety for the simplest case via the periods of $K3$ surfaces.
 \end{abstract}

\footnote[0]{Keywords: class fields ; $K3$ surfaces ; Shimura varieties ;  abelian varieties ; complex multiplication ; Hilbert modular functions ; quartic fields.  }
\footnote[0]{Mathematics Subject Classification 2010:  Primary 11G45 ; Secondary 14J28,  14G35, 11F46, 11G15, 11R16.}
\footnote[0]{Running head: Icosahedral invariants and class fields }
\setlength{\baselineskip}{14 pt}

\section*{Introduction}

The aim of this paper is to construct  class fields via $K3$ surfaces.
We obtain explicit class fields using the period mapping for a family of $K3$ surfaces parametrized by Klein's icosahedral invariants of \cite{Klein}.

In number theory, it is very interesting to  find good special functions whose special values generate class fields (cf. Hilbert's 12th problem).
Among such functions, the elliptic $j$-function is the most famous one.
For an elliptic curve $E(g_2:g_3): y^2=4x^3-g_2 x -g_3$, the $j$-function is defined by the weighted quotient $j(z)=\frac{g_2^3}{g_2^3-27 g_3^2}$, where $z\in \mathbb{H}$ is the period of $E(g_2:g_3)$.
If $E(g_2:g_3)$ has complex multiplication by an imaginary quadratic field $K$, it is well-known that $K=\mathbb{Q}(z)$ and the absolute class field of  $K $ is given by $K(j(z))$ (Kronecker's Jugendtraum).

In the 20th century, A. Weil,  G. Shimura,  Y. Taniyama and many other researchers  studied class fields using the modular functions coming from the moduli of abelian varieties with complex multiplication.
Their theory is celebrated.
However, 
since higher dimensional abelian varieties are defined by complicated equations, 
such modular functions are given by complicated forms (see [Mum] and [Gj]). 
So,  it is much more difficult than the classical story of the elliptic $j$-function to construct explicit class fields via abelian varieties.

In this paper, we give a construction of class fields using not only abelian varieties but also $K3 $ surfaces.
A $K3$ surface is a $2$-dimensional complex surface  with the unique holomorphic $2$-form up to a constant factor.
So,  a $K3$ surface can be regarded as a $2$-dimensional counterpart of an elliptic curve.
Moreover, algebraic $K3$ surfaces are sometimes defined by very  simple  equations.
In this paper, we consider the $K3$ surface
$
S(\mathfrak{A}:\mathfrak{B}:\mathfrak{C}) :  z^2=x^3-4(4y^3 -5 \mathfrak{A} y^2) x^2 + 20 \mathfrak{B} y^3 x +\mathfrak{C} y^4,
$
where $\mathfrak{A},\mathfrak{B}$ and $\mathfrak{C}$ are Klein's icosahedral invariants.
The period of  $S(\mathfrak{A}:\mathfrak{B}:\mathfrak{C}) $ is given by a point $(z_1,z_2)\in \mathbb{H} \times \mathbb{H}$.
 The weighted quotients $ X=\frac{\mathfrak{B}}{\mathfrak{A}^3},Y=\frac{\mathfrak{C}}{\mathfrak{A}^5}$ can be regarded as 
 the Hilbert modular functions on $\mathbb{H}\times \mathbb{H}$ for $\mathbb{Q}(\sqrt{5}),$
 which is the real quadratic field with the smallest discriminant.
We will see that $X,Y$ functions  generate  class fields over quartic CM fields (Theorem \ref{ThmCF}).
So, we can regard the pair of our $X$, $Y$  functions as a very natural and simple counterpart of the elliptic $j$-function (see Table 1).
Here, we note that the proof of this result is based on  techniques for elliptic $K3$ surfaces.

\begin{table}[h]
\center
\begin{tabular}{ccc}
\toprule
  & Classical Theory  & Our Story  \\
 \midrule
Variety &Elliptic curve $E(g_2:g_3)$ & $K3$ surface $S(\mathfrak{A}:\mathfrak{B}:\mathfrak{C})$ \\
 Modular Function& $j=\frac{g_2^3}{g_2^3-27 g_3^2}$ on $\mathbb{H}$  & $X=\frac{\mathfrak{B}}{\mathfrak{A}^3}$, $Y=\frac{\mathfrak{C}}{\mathfrak{A}^5}$   on $\mathbb{H}\times \mathbb{H}$ \\ 
  Totally Real Field & $\mathbb{Q}$ & $\mathbb{Q}(\sqrt{5})$  \\
 Class Field &   $K(j)/K$ & $K^*(X,Y)/K^*$   \\
 \bottomrule
\end{tabular}
\caption{The elliptic $j$-function and our $X,Y$ functions}
\end{table}

This paper is organized as follows.

In Section 1, we recall class fields from abelian varieties with complex multiplication.
We use these results in our argument.

In Section 2,
we will see that $X,Y$ functions generate the unramified class fields  $K^*(X,Y)/K^*$ for certain quartic CM fields $K^*$ (Theorem \ref{ThmCF}).
The $X, Y$ functions are given by  the inverse correspondence of the period mapping of our $K3$ surface $S(\mathfrak{A}:\mathfrak{B}:\mathfrak{C}).$
The Shioda-Inose structure connects the period mapping of $K3 $ surfaces and the moduli of abelian surfaces.
Moreover, we obtain an explicit expression of the canonical model of the Hilbert modular surface as a Shimura variety for the real quadratic field of the smallest discriminant (Theorem \ref{ThmCanonical}). 
We note that such canonical models were firstly studied  in \cite{S67}.
But, it is non-trivial to obtain an explicit expression for them.
In  Section 2, we  obtain an explicit model derived from the period mapping of our $K3$ surfaces. 
This result is given as a compilation of the previous works \cite{NaganoTheta}, \cite{NaganoKummer} and \cite{NaganoShimura}.

By the way, 
it is difficult
to determine the Galois group for a class field, when the Galois group is not equal to the ideal class group.
In fact, the Galois groups for our class fields $K^*(X,Y)/K^*$ do not always coincide with the ideal class groups of $K^*$.
In Section 3, we determine the structure of the Galois groups ${\rm Gal}(K^*(X,Y)/K^*)$ for cyclic extensions $K/\mathbb{Q}$ (Theorem \ref{ThmHantei}).
Our result relies upon the result of Shimura \cite{Shimura1} for cyclic extensions.
The results in this section may be applied to the theory of cryptography,
because explicit and detailed constructions of class fields are useful in that theory.

Our story gives a visible model of  the sophisticated theory  of class fields  and Shimura varieties.
The author expects that 
some explicit algebraic $K3$ surfaces (for example, those studied in \cite{CD}, \cite{EK}, etc.)  have direct applications in algebraic number theory via periods.
The argument in this paper gives a prototype of such mathematical approaches.
Moreover, $K3$ surfaces have some merits, for they are studied from various viewpoints.
For example,  $K3$ surfaces  are well-studied  in  mirror symmetry.
In fact, our  $X$, $Y$ functions are closely related to mirror symmetry (see Remark \ref{mirrorRem}).
The author believes that our story  concretely gives a new non-trivial relation between   number theory and mirror symmetry.

\section{Abelian varieties with complex multiplication and class fields}

In this section, we survey the theory of abelian varieties with complex multiplication.
For detail,  see  Shimura \cite{Shimura97}.
The results in this section shall be used in Section 2 and 3.

\subsection{CM fields and CM types}

Let $K_0$ be a totally real number field such that  $[K_0:\mathbb{Q}]=n$.
A totally imaginary quadratic extension $K$ of $K_0$ is called a CM field.
Let $\mathfrak{O}_K$ be the ring of integers of $K$.
We have $2n $ embeddings $\varphi_1,\cdots, \varphi_{2n}:K\hookrightarrow \mathbb{C}$.
We can suppose that 
$\varphi_1|_{K_0},\cdots, \varphi_{n}|_{K_0}$ give distinct $n$ embeddings $K_0\hookrightarrow \mathbb{R}$
and $\varphi_{n+j}$ $(j=1,\cdots,n)$ is complex conjugate of $\varphi_j$.
We call $(K,\{\varphi_1,\cdots,\varphi_n\})=(K,\{\varphi_j\})$ a CM type. 

For an  abelian variety $A$ defined over $\mathbb{C}$,
let ${\rm End}(A)$ be the ring of endomorphisms of $A$ and
 ${\rm End}_0(A)={\rm End}(A) \otimes_\mathbb{Z} \mathbb{Q}$.
 For an $n$-dimensional abelian variety $A$ and an embedding $\iota:K\hookrightarrow {\rm End}_0(A) $ such that 
 $\iota \circ id_K=id_{\rm{End}_0(A)}$,
 we call $(A,\iota)$ an abelian variety of type $K$.
Moreover,
for the analytic representation $S$ of ${\rm End}_0(A) $ and the complex conjugate $\rho$,
if $S$ ($S^\rho$, resp.) is equivalent to the direct sum of $\varphi_1,\cdots,\varphi_n$ ($\varphi_{n+1},\cdots,\varphi_{2n}$, resp.),
$(A,\iota)$ is called an abelian variety of type $(K,\{\varphi_j\})$.
If $(A_1,\iota_1)$ and $(A_2,\iota_2) $ are abelian varieties of  the same type,
$A_1$ and $A_2$ are isogenous to each other.

From a given CM type,
we can obtain a corresponding abelian variety as follows.
For $\alpha\in K$,
we set
$
u(\alpha)={}^t (  \alpha^{\varphi_1} ,  \cdots ,  \alpha^{\varphi_n} ) \in \mathbb{C}^n.
$
Let
$M\subset K $ be a free $\mathbb{Z}$-module of rank $2n$.
Let $\alpha_1,\cdots,\alpha_{2n}$ be a system of basis of $M$.
Then, the vectors $u(\alpha_1),\cdots , u(\alpha_{2n})$ span a lattice $\Lambda=\Lambda(M)$ of $\mathbb{C}^n$.
Take
$\zeta\in K$
 such that
 $K=K_0(\zeta)$,
 $\zeta^2\in K_0,$  $-\zeta^2$ is totally positive and  
$
{\rm Im}(\zeta^{\varphi_j}) >0$  for $j\in \{1,\cdots,n\}$.
Set
\begin{align}\label{RiemannShimura}
E(z,w)=\sum_{j=1}^n \zeta^{\varphi_j} (z_j\overline{w_j} - \overline{z_j}w_j)
\end{align}
for $z=(z_1,\cdots,z_n),w=(w_1,\cdots,w_n) \in \mathbb{C}^n$. 
We can check that
$E(z,\sqrt{-1}w)= -\sqrt{-1}\sum_{j=1}^n \zeta^{\varphi_j} (z_j\overline{w_j} + \overline{z_j}w_j)$
is symmetric and
positive non-degenerate.
We obtain
$
E(u(\alpha),u(\beta))={\rm Tr}_{K/\mathbb{Q}} (\zeta \alpha \beta^\rho)
$
for $\alpha,\beta \in K$.
We can find a positive integer $\nu$ such that
 $\nu E(z,w)$  defines a non-degenerate Riemann form on $\mathbb{C}^n/\Lambda(M)$.
Therefore, the complex torus $\mathbb{C}^n/\Lambda$ gives an n-dimensional abelian variety $A=A(M)$.
For $\alpha\in K$, set 
$
q(\alpha)=
\begin{pmatrix}
\alpha^{\varphi_1} & & 0 \\
&\cdots& \\
0&& \alpha^{\varphi_n}
   \end{pmatrix}.
$
The linear transformation given by $q(\alpha)$ defines $\iota (\alpha)\in {\rm End}_0(A(M)).$
So, $(A(M),\iota)$ is an abelian variety of type $(K,\{\varphi_j\})$.

The above $(A(M),\iota)$ satisfies $\iota (K) \subset {\rm End}_0(A(M))$.
However, we cannot assure that
\begin{align}\label{PrincipalA}
\iota (\mathfrak{O}_K) = {\rm End}(A(M)) \cap \iota (K).
\end{align}

\begin{df}\label{DefPrin}
 We call the pair $(A,\iota)$ satisfying
  (\ref{PrincipalA})  an abelian variety of CM type $(\mathfrak{O}_K,\{\varphi_j\})$.
\end{df}

\begin{rem}
In \cite{Shimura97},
$(A,\iota)$ in the above definition is called a principal abelian variety.
But, in this paper, we will consider principally polarized abelian varieties.
In order to avoid confusion, we use the terminology of the above definition.
\end{rem}

\begin{prop}\label{PropPrincipal} (\cite{Shimura97}, Section 7)
The abelian variety $(A(M),\iota)$ is of CM type $(\mathfrak{O}_K,\{\varphi_j \})$ if and only if $M$ is a fractional ideal of $K$.
\end{prop}

Let us consider 
abelian varieties $(A,\iota)=(A(\mathfrak{a}),\iota)$ coming from a fractional ideal $\mathfrak{a}$ of $K.$ 
Take  $(A,\Theta,\iota)$,
where $\Theta$ is a polarization  given by the Riemann form of (\ref{RiemannShimura}). 
We call $(A,\Theta,\iota)$ a polarized abelian variety  of type $(K,\{\varphi_j\};\zeta,\mathfrak{a})$, or of type $(\zeta,\mathfrak{a})$.
Let $\mathfrak{C}_0(K)$ be the abelian group consisting of all pairs $\{b,\mathfrak{c}\}$,
where $b\in K_0$ is totally positive and $\mathfrak{c}\in I_K$.
Here, the multiplication is given by 
$
\{b_1,\mathfrak{c}_1\}\{b_2,\mathfrak{c}_2\}=\{b_1b_2,\mathfrak{c}_1\mathfrak{c}_2\}.$
Set
$
\mathfrak{C}(K)=\mathfrak{C}_0(K) /\{ \{x x^\rho, x\mathfrak{O}_K\}| x\in K^\times\}.
$
We denote the coset of $\{b,\mathfrak{c}\}$ by
$
(b,\mathfrak{c}).
$
Suppose $P_1=(A_1,\Theta_1,\iota_1)$ ($P_2=(A_2,\Theta_2,\iota_2)$, resp.) be a polarized abelian variety of type $(K,\{\varphi_j\};\zeta_1,\mathfrak{a}_1)$ $((K,\{\varphi_j\};\zeta_2,\mathfrak{a}_2),$ resp.). 
Setting 
$
(b,\mathfrak{c})=(\zeta_1^{-1} \zeta_2, \mathfrak{a}_2^{-1} \mathfrak{a}_1),
$
 we put
$
(P_2:P_1) = (b,\mathfrak{c}).$
We remark that 
$(A_1,\Theta_1)$ and $(A_2,\Theta_2)$
 are isomorphic  
if and only if
$(P_2:P_1)=(1,\mathfrak{O}_K)$.

\subsection{Reflex}

\begin{df}\label{DfPrimitive}
Let $(K,\{\varphi_j\})$ be a CM type and $(A,\iota)$ be a corresponding abelian variety.
If $A$ is a simple abelian variety,
we call 
$(K,\{\varphi_j\})$
primitive.
\end{df}

\begin{prop} \label{PropReflex} (\cite{Shimura97}, Section 8)
Let $(K;\{\varphi_1,\cdots,\varphi_n\})$ be a CM type.
Let $L$ be a Galois extension of $\mathbb{Q}$ containing $K$.
Set $S$ be the set of all the elements of ${\rm Gal} (L/\mathbb{Q})$ inducing some $\varphi_j$  on $K_0$.
Set $
S^*=\{\sigma^{-1}|\sigma\in S\}$
and 
$
 H^*=\{ \gamma\in {\rm Gal}(L/\mathbb{Q})| \gamma S^*=S^*\}.
$
There exists the subfield $K^*$ of $L$ corresponding to the subgroup $H^* $ of ${\rm Gal} (L/\mathbb{Q})$.
Let $\{\psi_k\}$ be the set of all the embeddings $K^*\hookrightarrow \mathbb{C}$ coming from $S^*$.
Then, $(K^*,\{\psi_k\})$ is a primitive  CM type.
Moreover, $(K^*,\{\psi_k\})$ is determined only by $(K,\{\varphi_j\})$.
\end{prop}

\begin{df}\label{ReflexDef}
For a CM type $(K,\{\varphi_j\})$, 
the primitive CM type
$
(K^*;\{\psi_k\})
$
in Proposition \ref{PropReflex} is called the reflex of $(K,\{\varphi_j\})$.
\end{df}

\begin{exap}\label{ExapS} (see \cite{Shimura97}, 8.4)
Let $K_0$ be a real quadratic field and $K$ be a totally imaginary quadratic extension of $K_0$.
Namely, let $K$ be a CM field of degree $4$.
Take
$\zeta\in K$
such that
$-\zeta^2 \in K_0$ is  totally positive.
We can take $\varphi_1 ,\varphi_2 : K \hookrightarrow \mathbb{C}$ such that $\varphi_1=id$ and $\varphi_2=\varphi$ satisfying
${\rm Im }(\zeta^{\varphi})>0.$
Then, 
$
(K,\{\varphi_1,\varphi_2\})
$
is a CM type.
We have the following three cases.

(i) Suppose ${\rm Gal}(K/\mathbb{Q})\simeq (\mathbb{Z}/2\mathbb{Z})^2$. 
Letting the group $S$ in Proposition \ref{PropReflex} be $S=\{id,\sigma\}$,
we have 
${\rm Gal} (K/\mathbb{Q})=\{id, \sigma,\rho ,\sigma\rho\}$,
where $\rho $ is the complex conjugation.
In this case, the CM type $(K,\{id, \sigma\})$ is not primitive. 
The reflex $K^*$ is an imaginary quadratic field.

(ii)
Suppose ${\rm Gal}(K/\mathbb{Q})\simeq (\mathbb{Z}/4\mathbb{Z})$.
Letting the group $S$ in Proposition \ref{PropReflex} be $S=\{id,\sigma\}$,
we have 
${\rm Gal} (K/\mathbb{Q})=\{id, \sigma, \sigma^2=\rho ,\sigma^3\}$,
where $\rho $ is the complex conjugation.
In this case, the CM type 
$
 (K,\{1,\sigma\})
$
is primitive.
The reflex  of $(K,\{id, \sigma\})$ is given by
$
(K^*,\{\psi_k\})=(K,\{1,\sigma\}).
$

(iii) Suppose $K$ is not normal over $\mathbb{Q}$.
Letting
$K_0=\mathbb{Q}(\sqrt{\Delta})$ $(\Delta>0)$,
we can assume that $-\xi^2=x+y\sqrt{\Delta}$ $(x,y\in\mathbb{Q})$ is totally positive. 
Set
$
\xi=\sqrt{-1} \sqrt{x+y\sqrt{\Delta}} $ and 
$\xi^\varphi=\sqrt{-1} \sqrt{x-y\sqrt{\Delta}}$.
Since $K $ is not normal, $\xi^\varphi \not \in K$. So, the Galois closure $L$ is given by $L=\mathbb{Q}(\xi,\xi^\varphi)$.
Set
$\sigma:(\xi,\xi^\varphi)\mapsto (\xi^\varphi,-\xi)$
and 
$ \tau:(\xi,\xi^\varphi)\mapsto (\xi^\varphi,\xi)$.
We can show that the Galois group ${\rm Gal}(L/\mathbb{Q})$ is generated by $\sigma $ and $\tau$.
Moreover,
the CM type $(K,\{1,\varphi\}) $ is primitive 
and its reflex
is given by
$
(\mathbb{Q} (\xi+\xi^\varphi),\{1,\sigma\tau\}).
$
\end{exap}

\subsection{The field of moduli of principally polarized abelian varieties}

Let $k_1$ and $k_2$ be number fields.
Let $\sigma : k_1 \rightarrow k_2$ be an embedding.
For an algebro-geometric object $W$ over $k_1$,
we denote by $W^\sigma$ the image $W $ under $\sigma$.
Then, $W^\sigma$ is defined over $k_2$.

Let $k_1$ be an algebraic number fields with ${\rm char}(k_1)=0$.
Let $(A,\Theta)$ be a polarized abelian variety defined over $k_1$.
Take an Galois extension $L$ of $k_1$ over $\mathbb{Q}$.
We take a subgroup
$
G_1=\{ \sigma \in {\rm Gal}(L/\mathbb{Q}) | (A,\Theta) \text{ is isomorphic to } (A^\sigma,\Theta^\sigma)  \}.
$
Let $M$ be the subfield of $L$ corresponding to the subgroup $G_1$.
We note that the field $M$ is uniquely determined by $(A,\Theta)$ and independent of the choice of the field of definition $k_1$ of $(A,\Theta)$
 and the Galois closure $L$.

\begin{df}\label{Df1.2}
The above field $M$ is called the field of moduli of $(A,\Theta)$.
\end{df}

\begin{rem} \label{DfWeil} 
(\cite{Weil}, see also  \cite{Shimura97})
Let ${\rm Aut}(A,\Theta)$ be the group of automorphisms of a polarized abelian variety $(A,\Theta)$.
The quotient variety
$A/{\rm Aut}(A,\Theta)$ is called the Kummer variety of $(A,\Theta)$ in the sense of Weil \cite{Weil}.
Let $M$ be the field of moduli of $(A,\Theta)$.
Then, the field of definition of the  Kummer variety of $(A,\Theta)$  contains  $M$.
Moreover, there exists a model $W$ of the Kummer variety of $(A,\Theta)$ such that the field of definition of $W$ coincides with $M$. 
\end{rem}

In the rest of this subsection, we survey the field of moduli of  principally polarized abelian surfaces.
Let $(A,\Theta)$ be a simple and principally polarized abelian surface.
Then, there exists a  curve $C$ of genus $2$ such that $A$ is equal to the Jacobian variety ${\rm Jac}(C)$ of $C$.
Suppose $C$ is given by
 $
 C: y^2=u_0 (x-x_1)(x-x_2)(x-x_3)(x-x_4)(x-x_5)(x-x_6).
 $
Then, we obtain the Igusa-Clebsch invariants $I_2,I_4,I_6$ and $I_{10}$ of $C$: 
\begin{align}
\begin{cases}
&I_2 = u_0^2\sum (12)^2(34)^2(56)^2 ,\quad 
I_4 = u_0^4\sum (12)^2(23)^2(31)^2(45)^2(56)^2(64)^2 ,\\
&I_6 =u_0^6\sum (12)^2(23)^2(31)^2(45)^2(56)^2(64)^2(14)^2(25)^2(36)^2 ,\quad I_{10}= u_0^{10}\prod_{i<j} (ij)^{2}. \\
\end{cases}
\end{align}
Here, $(jk) $ indicates $(x_j-x_k)$.

 Let $\mathcal{M}_2$ be the moduli space of  curves of genus $2$.
Let $\mathbb{P}(1:2:3:5)=\{(\zeta_1:\zeta_2:\zeta_3:\zeta_5)\}$ be the weighted projective space.
It is well-known that 
$
\mathcal{M}_2=\mathbb{P}(1:2:3:5) -\{\zeta_5=0\}.
$
In fact,
$(I_2:I_4:I_6:I_{10})$ gives a well-defined point of the moduli space $\mathcal{M}_2$.
We note that the moduli space $\mathcal{M}_2$ is a Zariski open set of the moduli space $\mathcal{A}_2$ of principally polarized abelian surfaces ($\mathcal{M}_2$ is the complement of the divisor given by  the points corresponding to the product of elliptic curves).

In \cite{Igusa},
Igusa defined arithmetic invariants $J_2,J_4,J_6$ and $J_{10}$ given by
\begin{align}\label{IgusaI}
 J_2=2^{-3} I_2, \quad  J_{10}= 2^{-12} I_{10},\quad
 J_{4}=2^{-5} 3^{-1} (4 J_2^2 -I_{4}),\quad
J_{6}=2^{-6} 3^{-2} (8J_2^3 -160 J_2 J_4 -I_6).
\end{align}
Due to (\ref{IgusaI}), we have an isomorphism between
the ring
$\mathbb{Q}[I_2,I_4,I_6,I_{10}]$ 
and
the ring
 $\mathbb{Q}[J_2,J_4,J_6,J_{10}]$.

For a curve $C$ of genus $2$,
set 
\begin{align}\label{IgusaModular}
m_1(C)=\frac{J_2^5}{J_{10}}, \quad m_2(C)=\frac{J_2^3 J_4}{J_{10}}, \quad m_3(C)=\frac{J_2^2 J_6}{J_{10}}.
\end{align}
They are called the absolute invariants.
Then, we have the following famous result (see \cite{S63} or \cite{FG}).

\begin{prop}\label{ModularIgusaProp}
(1) Let $C_1$ and $C_2$ be two curves of genus $2$ with $I_{10}\not=0$.
Then, the curve $C_1$ is isomorphic to $C_2$ if and only if $m_j(C_1)=m_j(C_2) $ $(j=1,2,3)$.

(2) Let $(A,\Theta)$ be a principally polarized abelian surface and $C$ be a curve of genus $2$ such that ${\rm Jac}(C)=A$.
Then, a subfield $M$ of the field of definition of $(A,\Theta)$ gives a field of moduli of $(A,\Theta)$ if and only if  the absolute invariants $m_j (C)$ $(j=1,2,3)$
 generate $M$ over $\mathbb{Q}$.
 \end{prop}

\subsection{Class fields from abelian varieties}

In this subsection, 
we give a survey of class fields derived from abelian varieties.

Take a cycle $\mathfrak{m}$ of $k.$
Let $I_{k}(\mathfrak{m})$ be the group of fractional ideals of $k$ which are prime to $\mathfrak{m}.$
Set $P_k(\mathfrak{m)}=\{(\alpha) | \alpha \equiv 1 \hspace{1mm}({\rm mod } \hspace{1mm} \mathfrak{m})\}$.
Let $k_1/k$ be a Galois extension.
Let $I_{k_1}(\mathfrak{m})$ be the group of fractional ideals of $k_1$ which are prime to $\mathfrak{m}$.
The image $N_{k_1/k}(I_{k_1}(\mathfrak{m}))$ is a subgroup of $I_k(\mathfrak{m})$.
Then, we have the fundamental inequality
$
[I_k(\mathfrak{m}): P_k(\mathfrak{m}) N_{k_1/k} (I_{k_1}(\mathfrak{m}))]\leq [k_1:k].
$

\begin{df}\label{HCG}
Let $H(\mathfrak{m})$ be a subgroup of $I_k(\mathfrak{m})$ such that 
$
P_k(\mathfrak{m} ) \subset H(\mathfrak{m}) \subset I_k(\mathfrak{m}).
$
If a Galois extension $k_1/k_0$ satisfies
$
[I_k(\mathfrak{m}):H(\mathfrak{m})]=[k_1:k]
$
and 
$
H(\mathfrak{m})=P_k(\mathfrak{m}) N_{k_1/k} (I_{k_1}(\mathfrak{m})),
$
then $k_1/k$ is called the class field corresponding to $H(\mathfrak{m})$.
\end{df}

Under the notion in  Definition \ref{HCG},
we have an isomorphism
$
I_k(\mathfrak{m})/H(\mathfrak{m}) \simeq {\rm Gal} (k_1/k)
$
that is given by the Artin symbol
$\mathfrak{a} \mapsto \displaystyle \Big(\frac{k_1/k}{\mathfrak{a}}\Big)$ (The Artin reciprocity law).
Let $\mathfrak{f}(k_1/k)$ be the conductor of $k_1/k$.
If $\mathfrak{f}(k_1/k)=(1)$, then $k_1/k$ is unramified.
Moreover,
for any abelian extension $k_1/k$,
there exists a group $H(\mathfrak{m})$
of Definition \ref{HCG} 
such that $k_1/k$ is the class field corresponding to $H(\mathfrak{m})$.

Let $(K,\{\varphi_j\})$ be a CM type and $(K^*,\{\psi_k\})$ be the corresponding reflex.
Let $P=(A,\Theta,\iota)$ be a polarized abelian variety of type $(K,\{\varphi_j\})$.
Let $L$ be the Galois closure of $K$ over $\mathbb{Q}$.
For $\mathfrak{a}\in I_{K^*}$, setting
 \begin{align}\label{Phi*}
   \mathfrak{a}^{\Phi^*}= \prod_{k} \mathfrak{a} ^ {\psi_k} .
 \end{align}
We can take  $g(\mathfrak{a})\in I_K$ such that
$
\mathfrak{O}_L g(\mathfrak{a})=\mathfrak{O}_L  \mathfrak{a}^{\Phi^*}.
$
We need  the subgroup $H_0$ of $I_{K^*}$ given by
\begin{align}\label{H0}
H_0=\{\mathfrak{a}\in I_{K^*}| g(\mathfrak{a})\in P_{K}, \text{ there exists } \mu\in K \text{ such that } g(\mathfrak{a})=(\mu), N(\mathfrak{a})=\mu\mu^\rho\}.
\end{align}
For $\sigma \in {\rm Gal}(\mathbb{C}/K^*)$, 
$P^{\sigma}=(A^\sigma,\Theta ^\sigma,\iota^\sigma)$
has the same type of $(K,\{\varphi_j\})$.
So, 
we obtain
$
(P^\sigma:P)=(b,\mathfrak{c})
$
for some totally positive  $b\in K_0$ and $\mathfrak{c}\in I_K$.
So, we have a homomorphism  $h:{\rm Gal}(\mathbb{C}/K^*)  \rightarrow \mathfrak{C}(K)$  given by $\sigma \mapsto (b,\mathfrak{c})$.
 $P$ is isomorphic to $P^\sigma$ if and only if $\sigma\in {\rm Ker}(h)$.
Letting $M$ be the field of moduli of $(A,\Theta)$,
${\rm Ker}(h)$ is the set of $\sigma$ which leaves the elements of $k_0=MK^*$ invariant.
So, 
the homomorphism
$h$  induces an embedding 
$
 \tilde{h}:{\rm Gal}(k_0/K^*) \hookrightarrow \mathfrak{C}(K).
$
Especially, 
$k_0/K^*$ is an abelian extension.
Hence, we can take an integral ideal $\mathfrak{m}_0$ of $K^*$
such that 
the Artin symbol  $\sigma:I_{K^*} (\mathfrak{m}_0) \rightarrow {\rm Gal}(k_0/K^*)$
given by $\mathfrak{a} \mapsto \displaystyle  \sigma(\mathfrak{a})=\Big(\frac{k_0/K^*}{\mathfrak{a}}\Big)$ 
induces the homomorphism 
$\tilde{h}\circ \sigma:I_{K^*}(\mathfrak{m}_0)\rightarrow \mathfrak{C}(K)$ with the form
$
\mathfrak{a} \mapsto (P^{\sigma(\mathfrak{a})}:P)=(N(\mathfrak{a}),\mathfrak{a}^{\Phi^*}).
$
By the definition (\ref{H0}),
we can see that ${\rm Ker}(\sigma)=H_0\cap I_{K^*} (\mathfrak{m}_0)$.
So,  the Artin reciprocity law implies 
$I_{K^*}(\mathfrak{m}_0)/(H_0\cap I_{K^*}(\mathfrak{m}_0)) \simeq {\rm Gal}(k_0/K^*)$.
Moreover, we can prove that
 $H_0\supset P_{K^*}=P_{K^*}((1))$
and the conductor  $\mathfrak{f}(k_0/K^*)$ is $(1)$.
Thus, 
the following theorem follows.

\begin{thm} \label{ThmClassField} (\cite{Shimura97}, Section 15)
Let  $(K,\{\varphi_j\})$ be a CM type and  $(K^*,\{\psi_\alpha\})$ be  the corresponding reflex.
Let $H_0$ be the subgroup of $I_{K^*}$ given in (\ref{H0}).
 For a polarized abelian variety $P=(A, \Theta,\iota)$  of type  $(K,\{\varphi_j\}; \zeta, \mathfrak{a})$ coming from $\zeta \in K$ and $\mathfrak{a}\in I_{K^*}$, 
let $M$ be the field of moduli of $(A,\Theta)$.
Then,   $k_0=M K^*$ is the unramified class field over $K^*$ corresponding to $H_0$.
\end{thm}

\section{Applications of $X,Y$ functions to class field theory}

 We study the class fields in Section 1.4 for the case of $n=2$ and $K_0=\mathbb{Q}(\sqrt{5})$,
 that is the simplest case of complex multiplication of abelian surfaces
 because the discriminant of the real quadratic field is the smallest.
 In this section, 
 we shall give an explicit construction of the class fields
 using the special values of the functions $X$ and $Y$ of (\ref{XY}).

 \subsection{Klein's Icosahedral invariants and Hilbert modular surfaces}

 Let $K_0=\mathbb{Q}(\sqrt{\Delta})$ be a real quadratic field with discriminant $\Delta$ 
 and $\mathcal{O}_{\Delta}$ be the ring of integers.
 If an abelian surface $A$ and an embedding $\iota : K_0 \hookrightarrow {\rm End}_0(A)$ satisfies $\iota (\mathcal{O}_{\Delta})\subset {\rm End}(A)$,
 then the pair $(A,\iota)$ is called an abelian surface with real multiplication by $\mathcal{O}_\Delta$.
 The moduli space of principally polarized abelian surfaces with real multiplication by $\mathcal{O}_\Delta$ is called the Humbert surface $\mathcal{H}_\Delta$.
 The Humbert surface $\mathcal{H}_\Delta$ is a divisor in $\mathcal{A}_2.$
 
 In our study,
 we focus on the case for the smallest discriminant $\Delta=5$.
 In this case, the corresponding ring of integers is $\mathcal{O}_5=\langle 1,\frac{1+\sqrt{5}}{2}\rangle_\mathbb{Z}$.
 The Humbert surface $\mathcal{H}_5$ is isomorphic to the symmetric Hilbert modular surface $\langle PSL(2,\mathcal{O}_5),\tau \rangle \backslash (\mathbb{H}\times \mathbb{H})$, where $\tau$ is the involution on $\mathbb{H}\times \mathbb{H}$ given by $(z_1,z_2)\mapsto (z_2,z_1)$ (see \cite{Geer}).

 \begin{prop}\label{ThVV}
(\cite{Hirzebruch})
(1)
A compactification $\overline{\langle PSL(2,\mathcal{O}_5),\tau \rangle\backslash (\mathbb{H}\times \mathbb{H})}$ is isomorphic to the weighted projective plane 
$\mathbb{P}(1:3:5)={\rm Proj}(\mathbb{C}[\mathfrak{A},\mathfrak{B},\mathfrak{C}])$,
where $\mathfrak{A},\mathfrak{B}$ and $\mathfrak{C}$ are Klein's icosahedral invariants.

(2)
The Hilbert modular surface $\overline{PSL(2,\mathcal{O}_5)\backslash (\mathbb{H}\times \mathbb{H})}$ 
 is isomorphic to the hypersurface $V$
 of
 $\mathbb{P}(2:5:6:15)={\rm Proj}(\mathbb{C}[\mathfrak{A},\mathfrak{c},\mathfrak{B},\mathfrak{D}])$
 defined by the equation
 \begin{align}\label{KleinD}
 144\mathfrak{D}^2=-1728\mathfrak{B}^5 +720\mathfrak{A}\mathfrak{c}^2\mathfrak{B}^3 -80 \mathfrak{A}^2 \mathfrak{c}^4 \mathfrak{B} +64\mathfrak{A}^3(5\mathfrak{B}^2-\mathfrak{A}\mathfrak{c}^2)^2+\mathfrak{c}^6.
 \end{align}
 \end{prop}

\begin{rem}\label{RemVV}
In the above theorem,
the Hilbert modular surface
$\overline{PSL(2,\mathcal{O}_5) \backslash (\mathbb{H}\times \mathbb{H})} \simeq V$
is the double covering $\mathcal{R}$
of 
$\overline{\langle PSL(2,\mathcal{O}_5),\tau \rangle\backslash (\mathbb{H}\times \mathbb{H})} \simeq \mathbb{P}(1:3:5)={\rm Proj}(\mathbb{C}[\mathfrak{A},\mathfrak{B},\mathfrak{C}])$
corresponding to the involution $\tau$.
For any $(\mathfrak{A}:\mathfrak{B}:\mathfrak{C})$,
we have $\mathcal{R}^{-1}(\mathfrak{A}:\mathfrak{B}:\mathfrak{C})=\{(\mathfrak{A}:\mathfrak{c}:\mathfrak{B}:\mathfrak{D}), (\mathfrak{A}:-\mathfrak{c}:\mathfrak{B}:\mathfrak{D})\}\subset V$,
where $\mathfrak{c}^2= \mathfrak{C}$.
\end{rem}

\subsection{The family $\mathcal{F}$ of $K3$ surfaces and the period mapping}

In \cite{NaganoTheta},
the family $\mathcal{F}=\{S(\mathfrak{A}:\mathfrak{B}:\mathfrak{C})\}$ of $K3$ surfaces,
where
\begin{align}\label{S(ABC)}
S(\mathfrak{A}:\mathfrak{B}:\mathfrak{C}):  z^2=x^3-4(4y^3 -5 \mathfrak{A} y^2) x^2 + 20 \mathfrak{B} y^3 x +\mathfrak{C} y^4
\end{align}
 is studied in detail.
Here, $(\mathfrak{A}:\mathfrak{B}:\mathfrak{C})\in \mathbb{P}(1:3:5)-\{(1:0:0)\}$.
We note that we have an isomorphism $S(\mathfrak{A}:\mathfrak{B}:\mathfrak{C}) \mapsto S(\kappa^2 \mathfrak{A}:\kappa^6 \mathfrak{B}:\kappa^{10} \mathfrak{C})$ of elliptic $K3$ surfaces given by $(x,y,z)\mapsto (\kappa^6 x,\kappa^2 y, \kappa^9 z)$ for $\kappa\not = 0$.

The multivalued period mapping $\Phi$ for $\mathcal{F}$ has the form $\Phi:\mathbb{P}(1:3:5)-\{(1:0:0)\} \rightarrow \mathbb{H}\times \mathbb{H}$
given by
\begin{align}\label{z1z2}
(z_1,z_2)=\left( -\frac{ \int_{\Gamma_3}\omega + \frac{1-\sqrt{5}}{2}\int_{\Gamma_4}\omega}{ \int_{\Gamma_2} \omega},-\frac{ \int_{\Gamma_3}\omega +\frac{1+\sqrt{5}}{2}\int_{\Gamma_4} \omega}{ \int_{\Gamma_2}\omega}\right),
\end{align}
where $\omega$ is the unique holomorphic $2$-form up to a constant factor and 
$\Gamma_1,\cdots,\Gamma_4 $
are certain $2$-cycles on $S(\mathfrak{A}:\mathfrak{B}:\mathfrak{C})$.
In \cite{NaganoKummer}, the above integrals are explicitly given by double integrals on a weighted projective plane.
The inverse correspondence of $\Phi$ has an explicit expression by Hilbert modular functions as follows (for detail, see \cite{NaganoTheta}).
 Let us consider the holomorphic mapping $\mu_5:\mathbb{H}\times\mathbb{H} \rightarrow \mathfrak{S}_2$ given  by
\begin{align}\label{Mullerembedding}
(z_1,z_2) \mapsto  \frac{1}{2\sqrt{5}}\begin{pmatrix}  (1+\sqrt{5})z_1 -(1-\sqrt{5})z_2 & 2(z_1-z_2) \\ 2(z_1 -z_2) & (-1+\sqrt{5}) z_1 +(1+\sqrt{5}) z_2 \end{pmatrix}. 
\end{align}
We note that $\mu_5$ of (\ref{Mullerembedding}) gives a parametrization of the surface
\begin{align}\label{N5}
N_5=\{ \begin{pmatrix} \tau_1 & \tau_2 \\  \tau_2 &  \tau_3 \end{pmatrix} \in \mathfrak{S}_2 |  -\tau_1 + \tau_2 + \tau_3  =0\}.
\end{align}
It is known that  $pr(N_5)$, where $pr:\mathfrak{S}_2\rightarrow Sp(2,\mathbb{Z})\backslash \mathfrak{S}_2 =\mathcal{A}_2$ is the canonical projection,
coincides with the Humbert surface $\mathcal{H}_5$.
For $\Omega\in\mathfrak{S}_2$ and $a,b \in \{0,1\}^2$ with ${}^t a b\equiv 0 \hspace{1mm}({\rm mod}2)$,
set
$
\vartheta (\Omega;a,b)=\sum _{g\in \mathbb{Z}^2} {\rm exp}\Big(\pi \sqrt{-1} \big( {}^t\big( g+\frac{1}{2} a\big) \Omega \big( g+\frac{1}{2}a \big)+{}^tgb\big)\Big).
$
For $j\in\{0,1,\cdots,9\}$, we set
$
\theta_j (z_1,z_2)=\vartheta (\mu_5(z_1,z_2 ) ;a,b),
$
where the correspondence between $j$ and $(a,b)$ is given by Table 2.
\begin{table}[h]
\center
{\footnotesize
\begin{tabular}{lcccccccccc}
\toprule
$j$&$0$&$ 1$ &$2$ &$3$&$4$&$5$&$6$&$7$&$8$&$9$  \\
\midrule
${}^t a$& $(0,0)$ &$(1,1)$ &$(0,0)$&$(1,1)$&$(0,1)$&$(1,0)$&$(0,0)$&$(1,0)$&$(0,0)$ &$(0,1)$ \\
${}^t b$ &$(0,0)$&$(0,0)$&$(1,1)$&$(1,1)$& $(0,0) $& $(0,0)$& $(0,1)$& $(0,1)$ & $(1,0)$&$(1,0)$\\
\bottomrule
\end{tabular}
}
\caption{The correspondence between $j$ and $(a,b)$.}
\end{table}
Let $a\in \mathbb{Z}$ and $j_1,\cdots,j_r \in\{0,\cdots,9\}$. 
We set  $\theta_{j_1,\cdots,j_r}^a  = \theta_{j_1}^a \cdots \theta_{j_r} ^a$.
The following $g_2$ ($s_5,s_6,s_{15}$, resp.) is   a symmetric Hilbert modular form of weight $2$ ($5, 6,15,$ resp.) for  $\mathbb{Q}(\sqrt{5})$ (see M\"uller \cite{Muller}):
{\small
\begin{align}
\begin{cases}
&g_2= \theta_{0145}-\theta_{1279}-\theta_{3478}+\theta_{0268} +\theta_{3569}, \quad \quad s_5 = 2^{-6} \theta_{0123456789},\\
&s_6=2^{-8} (\theta_{012478}^2 +\theta_{012569}^2 +\theta_{034568}^2 + \theta_{236789}^2 +\theta_{134579}^2),\\
& s_{15}=-2^{-18} (\theta_{07}^9 \theta_{18}^5 \theta_{24} - \theta_{25}^9 \theta_{16}^5 \theta_{09} +\theta_{58}^9 \theta_{03}^5 \theta_{46} -\theta_{09}^9\theta_{25}^5\theta_{16} +\theta_{09}^9\theta_{16}^5\theta_{25} -\theta_{67}^9 \theta_{23}^5 \theta_{89} \\
&\quad\quad\quad\quad \quad
+\theta_{18}^9\theta_{24}^5 \theta_{07} -\theta_{24}^9 \theta_{18}^5 \theta_{07}
-\theta_{46}^9 \theta_{03}^5 \theta_{58} - \theta_{24}^9 \theta_{07}^5 \theta_{18}
-\theta_{89}^9\theta_{67}^5\theta_{23} -\theta_{07}^9\theta_{24}^5\theta_{18}\\
&\quad\quad\quad\quad \quad
+\theta_{89}^9\theta_{23}^5 \theta_{67} -\theta_{49}^9\theta_{13}^5 \theta_{57}
+\theta_{16}^9\theta_{09}^5\theta_{25} -\theta_{03 } ^9 \theta_{46}^5 \theta_{58}
+\theta_{16}^9\theta_{25}^5 \theta_{09} -\theta_{46}^9\theta_{58}^5\theta_{03}\\
&\quad\quad\quad\quad \quad
-\theta_{25}^9 \theta_{09}^5 \theta_{16} -\theta_{57}^9 \theta_{49}^5 \theta_{13}
+\theta_{67}^9\theta_{89}^5\theta_{23} +\theta_{58} ^9 \theta_{46}^5 \theta_{03}
+\theta_{57}^9  \theta_{13}^5 \theta_{49} -\theta_{23}^9 \theta_{89}^5 \theta_{67}\\
&\quad\quad\quad\quad \quad
+\theta_{18}^9 \theta_{07}^5 \theta_{24} +\theta_{03}^9 \theta_{58}^5 \theta_{46}
+\theta_{23}^9 \theta_{67}^5 \theta _{89} +\theta_{49}^9 \theta_{57}^5 \theta_{13}
-\theta_{13}^9 \theta_{57}^5 \theta_{49 } +\theta_{13}^9 \theta_{49}^5 \theta_{57}).
\end{cases}
\end{align}
}

\begin{prop}\label{ThetaRepn} (\cite{NaganoTheta} Theorem 4.1)
The inverse correspondence $\mathbb{H} \times \mathbb{H} \rightarrow \mathbb{P}(1:3:5)$ of the multivalued period mapping $\Phi$ is given by
\begin{align}\label{PCanonical}
(z_1,z_2) \mapsto (\mathfrak{A}:\mathfrak{B}:\mathfrak{C})=(g_2(z_1,z_2):2^5\cdot 5^2 s_6(z_1,z_2): 2^{10} \cdot 5^5 s_{10}(z_1,z_2)).
\end{align}
\end{prop}

 Setting $X=\frac{\mathfrak{B}}{\mathfrak{A}^3}, Y=\frac{\mathfrak{C}}{\mathfrak{A}^5}$,
 $(X,Y)$ gives a system of affine coordinates of  $\mathbb{P}(1:3:5)$.
 From (\ref{PCanonical}),
 we have the meromorphic functions
 \begin{align}\label{XY}
\displaystyle X(z_1,z_2)=2^5 \cdot 5^2 \cdot \frac{s_6(z_1,z_2)}{g_2^3(z_1,z_2)},
\quad \quad \quad 
\displaystyle Y(z_1,z_2)=2^{10} \cdot 5^5  \cdot\frac{s_{10}(z_1,z_2)}{g_2^5 (z_1,z_2)}
\end{align}
on $\mathbb{H} \times \mathbb{H}$. We call them $X,Y$ functions.

 \subsection{The Shioda-Inose structure of $\mathcal{F}$}

 Let $A $ be an abelian variety.
The minimal resolution ${\rm Kum}(A)$ of the quotient variety $A/\{id_A,-id_A\}$
 is called the Kummer surface.
We note that ${\rm Kum}(A)$ is an algebraic $K3$ surface.

 Let $S$ be an algebraic $K3$ surface.
Let $\omega$ be the unique holomorphic $2$-form on $S$ up to a constant factor.
If an involution $\nu$ on  $S$ satisfies
$\nu^* \omega =\omega$, we call $\mathcal{\nu}$ a symplectic  involution. 
Set $G=\{   id ,\nu \} \subset {\rm Aut}(S)$.
The minimal resolution $S'$ of the quotient variety $S/G$ is a $K3$ surface.
We have  the rational quotient mapping $\chi:S \dashrightarrow S'$.
We say that a $K3$ surface $S$ admits a Shioda-Inose structure if there exists a symplectic involution 
$\iota\in{\rm Aut}(S) $ with the rational quotient mapping $\chi: S\dashrightarrow S'$ such that $S'$ is a Kummer surface and $\chi_*$ induces a Hodge isometry ${\rm Tr}(S)(2)\simeq{\rm Tr } (S')$.

\begin{prop} \label{ThmNSI} (\cite{NaganoKummer} Section 2)
If a principally polarized abelian surface $(A,\Theta)$ has real multiplication by $\mathcal{O}_5$,
then there exists a  
 Hodge isometry
\begin{align}\label{ShiodaInose}
{\rm Tr}(A) \simeq  {\rm Tr}(S(\mathfrak{A}:\mathfrak{B}:\mathfrak{C}))
\end{align}
for some $(\mathfrak{A}:\mathfrak{B}:\mathfrak{C})\in \mathbb{P}(1:3:5)-\{(1:0:0)\}$ via the Shioda-Inose structure of $S(\mathfrak{A}:\mathfrak{B}:\mathfrak{C})$.
 \end{prop}

\begin{rem}
The corresponding Kummer surface $K(X,Y)$ coming from the Shioda-Inose structure with $\chi:S(\mathfrak{A}:\mathfrak{B}:\mathfrak{C}) \rightarrow K(X,Y)$   is given by the elliptic surface
 \begin{eqnarray}\label{K(X,Y)}
K(X,Y):  v^2=(u^2-2y^5)(u-(5 y^2-10X y+Y)).   
\end{eqnarray} 
 \end{rem}

 \subsection{The field of moduli via periods of $K3$ surfaces}

 Let $\mathfrak{S}_2$ be the Siegel upper half plane of degree $2$.
 Let $Sp(2,\mathbb{Z})$ be the symplectic group consisting of $4\times 4$ matrices.
 It is well known that the Igusa $3$-fold $\mathcal{A}_2$ of principally polarized abelian surfaces is given by the quotient space $Sp(2,\mathbb{Z})\backslash \mathfrak{S}_2$.
 Here, recall that the moduli space $\mathcal{M}_2={\rm Proj}(\mathbb{C}[I_2,I_4,I_6,I_{10}])$
 of genus $2$ curves is a Zariski open set of $\mathcal{A}_2.$

 \begin{prop}\label{ThmFOM}
Let $(A,\Theta)$ be a principally polarized abelian surface with real multiplication by $\mathcal{O}_{5}$. 
Let $S(\mathfrak{A}:\mathfrak{B}:\mathfrak{C})$ be the corresponding $K3$ surface under the Hodge isometry of (\ref{ShiodaInose}).
Then, the field of moduli of $(A,\Theta)$ is given by $\mathbb{Q}(X,Y)$. 
 \end{prop}

\begin{proof}
We note that, by Proposition \ref{ModularIgusaProp},  
the field of moduli $(A,\Theta)$ is given by $\mathbb{Q}(m_1,m_2,m_3)$.

Let us introduce the family
 $\mathcal{F}_{CD}=\{S_{CD}(\alpha:\beta:\gamma:\delta)\}$ of $K3$ surfaces,
 where $(\alpha,\beta,\gamma,\delta)\in \mathbb{P}(2:3:5:6)-\{\gamma=\delta=0\}$.
 The member $S_{CD}(\alpha:\beta:\gamma:\delta)$
 is defined by the explicit  equation
 \begin{align}\label{Mother}
S_{CD}(\alpha:\beta:\gamma:\delta): y^2 =x^3 +(-3 \alpha t^4 - \gamma t^5)x +(t^5 -2\beta t^6 + \delta t^7)
\end{align}
of  an elliptic $K3$ surface.
This $K3$ surface is due to Kumar \cite{Kumar} or Clingher and Doran \cite{CD}.
Here,  the moduli space
for $\mathcal{F}_{CD}$ is isomorphic to the Igusa 3-fold $\mathcal{A}_2=Sp(2,\mathbb{Z})\backslash \mathfrak{S}_2 $.
Then,
$(I_2:I_4:I_6:I_{10})\in \mathbb{P}(1:2:3:5)-\{I_{10}=0\}$
corresponds to 
a point 
$(\alpha:\beta:\gamma:\delta)\in \mathbb{P}(2:3:5:6)-\{\gamma=\delta=0\}$,
via the period mapping of $\mathcal{F}_{CD}$.
Namely, 
we have the following mapping over $\mathbb{Q}$:
\begin{align}\label{abcd}
 \displaystyle \alpha=\frac{1}{9} I_4 ,\quad
  \beta=\frac{1}{27} (-I_2 I_4 +3 I_6), \quad
 \displaystyle \gamma= 8 I_{10},  \quad
 \displaystyle \delta=\frac{2}{3} I_2 I_{10}.
\end{align}
Moreover,
$S_{CD}(\alpha:\beta:\gamma:\delta)$ has the Shioda-Inose structure
and there exists the Hodge isometry
\begin{align}\label{SIGeneric}
{\rm Tr}(A) \simeq {\rm Tr} (S_{CD}(\alpha:\beta:\gamma:\delta)),
\end{align}
where $A$ is a generic principally polarized abelian surface.
Under (\ref{SIGeneric}),
an principally polarized abelian surface corresponds to $S(\alpha:\beta:\gamma:\delta)$ (see \cite{CD} or \cite{NaganoShiga}). 

In \cite{NaganoShimura} Theorem 3.16,
we proved that an principally polarized abelian surface $A$
has real multiplication by $\mathcal{O}_5$
if and only if
$(\alpha:\beta:\gamma: \delta )$ is a point of $\Psi_5 (\mathbb{P}(1:3:5))$,
where $\Psi_5 : \mathbb{P}(1:3:5) \rightarrow \mathbb{P}(2:3:5:6) $ is given by
\begin{align}\label{(ABCDXY)}
\begin{cases}
\vspace{2mm}
&\alpha_5(\mathfrak{A}:\mathfrak{B}:\mathfrak{C})=\displaystyle \frac{25}{36} \mathfrak{A}^2,\quad
\beta_5 (\mathfrak{A}:\mathfrak{B}:\mathfrak{C})= \displaystyle  \frac{1}{2}\Big(- \frac{125}{108} \mathfrak{A}^3 + \frac{5}{4} \mathfrak{B} \Big),\\
&\gamma_5 (\mathfrak{A}:\mathfrak{B}:\mathfrak{C})= \displaystyle  \frac{1}{32} \mathfrak{C} , \quad
\delta_5 (\mathfrak{A}:\mathfrak{B}:\mathfrak{C})= \displaystyle \frac{25}{64} \mathfrak{B}^2 - \frac{5}{96} \mathfrak{A} \mathfrak{C}.
\end{cases}
\end{align}
Especially,
$\Psi_5 $ is defined over $\mathbb{Q}$.

From (\ref{abcd}) and (\ref{IgusaI}),
we have 
\begin{align}\label{Eqn1}
\mathbb{Q}[J_2,J_4,J_6,J_{10}]=\mathbb{Q}[\alpha,\beta,\gamma,\delta].
\end{align}
Moreover, if  a principally polarized abelian surface $(A,\Theta)$ has real multiplication by $\mathcal{O}_5$,
by virtue of the embedding $\Psi_5$ of (\ref{(ABCDXY)}),
we have 
\begin{align}\label{Eqn2}
\mathbb{Q}[\alpha,\beta,\gamma,\delta]=\mathbb{Q}[\mathfrak{A},\mathfrak{B},\mathfrak{C}].
\end{align}
Hence, from (\ref{IgusaModular}), (\ref{Eqn1}) and (\ref{Eqn2}), if $(A,\Theta)$ has real multiplication by $\mathcal{O}_{5}$,
we have $\mathbb{Q}(m_1,m_2,m_3)=\mathbb{Q}(X,Y)$.
\end{proof}

\begin{rem}
Due to Proposition \ref{ThmFOM}, 
the  equation (\ref{K(X,Y)}) of the Kummer surface gives an explicit model of the Kummer variety $W$
 in the sense of Remark \ref{DfWeil}.
To see this,
we note that
 the torsion part $E_{{\rm End}(A)}^{tor}$  of the group $E_{{\rm End}(A)}$ of all the units of ${\rm End}(A)$
contains ${\rm Aut}(A,\Theta)$ (see \cite{Shimura97}).
  If a unit $u$ in the ring $\mathcal{O}_5$ is  a torsion element, 
  then $u=1$ or $u=-1$.
This implies that  ${\rm Aut}(A,\Theta)=\{id_A,-id_A\}$.
\end{rem}

Let us apply the argument in Section 1.1 to our case.
Let $K$ be  an imaginary quadratic extension over $K_0=\mathbb{Q}(\sqrt{5})$. 
Let $(K,\{id,\sigma \})$ be a CM type of $K$. 
We set $u(\alpha)=\begin{pmatrix} \alpha \\ \alpha^\sigma \end{pmatrix} \in \mathbb{C}^2$ for $\alpha \in K$.
Take $\mathfrak{a}\in I_K$.
For a system $\{\alpha_1,\cdots,\alpha_4\}$ of the basis of  $\mathfrak{a}$,
the $2\times 4$ matrix
$(u(\alpha_1)\cdots u(\alpha_4))$ gives a lattice $\Lambda(\mathfrak{a})$ and we have a complex torus $\mathbb{C}^2/\Lambda(\mathfrak{a})=A(\mathfrak{a}).$
If the alternating form $E$ of (\ref{RiemannShimura}) given by $\zeta \in K$ gives a polarization $\Theta$ on $A(\mathfrak{a})$,
$(A(\mathfrak{a}),\Theta)$ is a polarized abelian surface
and $(A(\mathfrak{a}), \iota)$ is of CM type $(\mathfrak{O}_K,\{\varphi_j\})$ in the sense of Definition \ref{DefPrin} (see Proposition \ref{PropPrincipal}).

\begin{prop}\label{PropCF}
In the above notation,
suppose $\Theta$ is a principal polarization.

{\rm (1)} 
There exists a basis of $\mathbb{C}^2$ such that the matrix $(u(\alpha_1)\cdots u(\alpha_4))$
is expressed in the form
\begin{align}\label{periodA}
(u(\alpha_1)u(\alpha_2)u(\alpha_3)u(\alpha_4))=\begin{pmatrix}\tau_1 & \tau_2& 1 &0  \\ \tau_2 & \tau_3& 0 &1  \end{pmatrix}
\quad \Big(\begin{pmatrix}\tau_1 & \tau_2  \\ \tau_2 & \tau_3  \end{pmatrix} \in\mathfrak{S}_2, \quad \tau_1=\tau_2+\tau_3 \Big).
\end{align}

{\rm (2)} 
Put
\begin{align}\label{z10,z20}
z_1^0=\frac{\tau_2+\sqrt{5}\tau_2+2\tau_3}{2}, \quad z_2^0=\frac{\tau_2-\sqrt{5} \tau_2+2\tau_3}{2}.
\end{align}
Then, the field of moduli of 
$(A,\Theta)$ is given by $\mathbb{Q}(X(z_1^0,z_2^0),Y(z_1^0,z_2^0))$, where
\begin{align}\label{OurthetaSpecial}
\displaystyle X(z_1^0,z_2^0)=2^5 \cdot 5^2 \cdot \frac{s_6(z_1^0,z_2^0)}{g_2^3(z_1^0,z_2^0)},
\quad \quad \quad 
\displaystyle Y(z_1^0,z_2^0)=2^{10} \cdot 5^5  \cdot\frac{s_{10}(z_1^0,z_2^0)}{g_2^5 (z_1^0,z_2^0)}.
\end{align}

\end{prop}

 \begin{proof}
 (1) Under the assumption,
 $A$ satisfies
 $\mathfrak{O}_K = {\rm End}(A)$.
 So, $(A,\Theta)$ gives a principally polarized abelian surface with real multiplication by $\mathcal{O}_{5}$.
 Recall that the Humbert surface $\mathcal{H}_5$ is given by $pr (N_5)$,
 where $N_5$ is given in (\ref{N5}).
 Then, the assertion follows.

 (2) 
 For a principally polarized abelian surface $(A,\Theta)$ with real multiplication by $\mathcal{O}_{5}$,
 there exists a $K3$ surface $S(\mathfrak{A}:\mathfrak{B}:\mathfrak{C})$ via the correspondence (\ref{ShiodaInose}).
 When $(A,\Theta)$ is given by the period matrix (\ref{periodA}),
 due to Proposition \ref{ThetaRepn}, \ref{ThmNSI}, \ref{ThmFOM} and 
 the modular embedding (\ref{periodA}),
we have the assertion.
\end{proof}

From Theorem \ref{ThmClassField} and the above proposition (2),
 the following theorem  holds.

\begin{thm}\label{ThmCF}
In the above notation,  let  $K^*$ be the reflex of the CM type $(K,\{id,\varphi\})$.
Then, the field $K^* (X(z_1^0,z_2^0),Y(z_1^0,z_2^0))$ is the unramified class field over $K^*$ for the group $H_0$ of (\ref{H0}).
\end{thm}

 \begin{rem}
 Lauter and Yang \cite{LY} introduced Gundlach invariants
 based on the work of Gundlach \cite{Gundlach}.
 By combining the arguments of \cite{Gundlach}, \cite{Muller} and \cite{LY},
we can obtain an expression of the pair of our $X, Y$ functions in terms of
Gundlach invariants.
In fact, we can give another proof via Gundlach invariants of Theorem  \ref{ThmCF}.
The reason why we give the proof based on  the periods of $K3$ surfaces is that our argument has the following good points.
 
 \begin{itemize}
 \vspace{-3mm}
 
 \item In our research, 
 $X,Y$
 functions are Hilbert modular functions of the coordinates $(z_1,z_2)$ of (\ref{z1z2})
 given by the period of $K3$ surfaces.
 In Theorem \ref{ThmCanonical} (the paper \cite{NaganoShimura}, resp.),
 explicit models of Shimura varieties (Shimura curves, resp.) are studied,
 using $X,Y$ functions in the above sense.
   
   \vspace{-3mm}

 \item
 It is expected that  the family $\{S_{CD} (\alpha:\beta:\gamma:\delta)\}$ 
 of (\ref{Mother})
 are useful to construct class fields and explicit models of Shimura varieties  in many other cases.
 The argument in this paper gives a prototype of such applications of $K3$ surfaces to number theory.
 \end{itemize}
 \end{rem}

 \begin{rem}\label{mirrorRem}
 $K3$ surfaces coming from toric varieties are important object in mirror symmetry.
Our K3 surface $S(\mathfrak{A}:\mathfrak{B}:\mathfrak{C})$, in fact,
is given by a hypersurface of a toric $3$-fold.
In the study of mirror symmetry,
toroidal compactifications of moduli spaces are very important.
In the paper \cite{HNU},
a toroidal compactification from the viewpoint of mirror symmetry was given in terms of $X,Y$ functions.
So, Theorem \ref{ThmCF} gives an explicit example connecting number theory and mirror symmetry.
   \end{rem}

\subsection{The canonical model of a Shimura variety}

Let $K_0 $ be a totally real field of degree $n$ and $B $ be a quaternion algebra over $K_0$
such that $B\otimes_{K_0} \mathbb{R} \simeq M_2(\mathbb{R})^n$.
Let $\mathfrak{O}_B$ be a maximal order  of $B$.
We set 
$\Gamma_1(\mathfrak{O}_B)=\{\gamma\in \mathfrak{O}_B| {\rm Nr}_{B/K_0}(\gamma)=1\}$.
Then,
$\Gamma_1(\mathfrak{O}_B)$ acts on $\mathbb{H}^n.$
Let $K$ be a totally imaginary quadratic extension over $K_0$
with an $K_0$-linear embedding $f:K\hookrightarrow B$.
We can see that there exists the unique fixed point $z_0\in \mathbb{H}^n$
of $f(K)\subset B\otimes_{K_0} \mathbb{R} \simeq M_2(\mathbb{R})^n$ (see \cite{S67} 2.6).
Shimura \cite{S67} studied such quotient spaces 
$\Gamma_1(\mathfrak{O}_B) \backslash \mathbb{H}^n$ deeply 
and 
defined 
the canonical models.
They are the origin of  Shimura varieties.
 
\begin{thm} (\cite{S67} 9.3)
 In the above notation, 
 there exist a Zariski open subset $V_S$ of a projective variety and a holomorphic mapping $\varphi_S:\mathbb{H}^n\rightarrow V_S$
 satisfying the following conditions.
 
 (S-i) $V_S$ is normal and defined over $\mathbb{Q}$.
 
 (S-ii) $\varphi_S$ gives a biregular isomorphism $\Gamma_1(\mathfrak{O}_B)\backslash \mathbb{H}^n \simeq V_S$
 
 (S-iii) Suppose $f$ satisfies  $f(\mathfrak{O}_K) \subset \mathfrak{O}_B$.
 Then, $K^* (\varphi_S(z))/K^*$ gives the class field for the group $H_0$ of (\ref{H0}).
 \end{thm}
 
 The pair
 $(V_S,\varphi)$ in the above theorem
  is called the canonical model of 
 $\Gamma_1(\mathfrak{O}_B)\backslash \mathbb{H}^n$.

 If $K_0$ is a real quadratic field $\mathbb{Q}(\sqrt{\Delta})$
 and
 $B=M_2(\mathbb{Q}(\sqrt{\Delta}))$,
 then 
 $\mathfrak{O}_B=M_2(\mathcal{O}_{\Delta})$ and 
 $\Gamma_1(\mathfrak{O}_B)\backslash (\mathbb{H}\times\mathbb{H})$ coincides with the Hilbert modular surface $PSL(2,\mathcal{O}_\Delta) \backslash  (\mathbb{H}\times \mathbb{H}) $.
 In this section,
 as an application of Theorem \ref{ThmCF},
 we  obtain the
canonical model  
for the  simplest case.

 \begin{thm}\label{ThmCanonical}
 Let $V$ be the variety of Proposition \ref{ThVV} (2).
 The holomorphic mapping $\varphi:\mathbb{H} \times \mathbb{H} \rightarrow V$ defined by 
 \begin{align}\label{ourphi}
(z_1,z_2)\mapsto (\mathfrak{A}:\mathfrak{c}:\mathfrak{B}:\mathfrak{D})=(g_2(z_1,z_2):2^5 \cdot 5^2 \sqrt{5} s_5(z_1,z_2): 2^5 \cdot 5^2 g_6(z_1,z_2): 2^{13} \cdot 5^5 \cdot 3^{-1} s_{15} (z_1,z_2)) 
 \end{align}
 gives the canonical model $(V,\varphi)$ of the $\Gamma_1(\mathfrak{O}_B)\backslash (\mathbb{H}\times \mathbb{H})$ for $K_0=\mathbb{Q}(\sqrt{5})$ and $B=M_2(K_0).$
 \end{thm}
 
 \begin{proof}
 
 From Proposition \ref{ThetaRepn}, 
we can see that  $(V,\varphi)$ of (\ref{ourphi})
satisfies {\em (S-i)} and {\em (S-ii)}.
Here, we note that the coefficient $2^{13} \cdot 5^5 \cdot 3^{-1}$ of $s_{15}$
was calculated in the proof of \cite{NaganoTheta} Theorem 4.1.
In this proof, we shall show that   $(V,\varphi)$ satisfies {\em (S-iii)}.

For the real quadratic field $\mathbb{Q}(\sqrt{\Delta})$ of discriminant $\Delta$,
let us consider a principally polarized abelian variety $(A,\Theta)$ with 
$\tilde{\iota}: \mathbb{Q}(\sqrt{\Delta})\hookrightarrow {\rm End}_0(A)$. 
We note that we have two choices of $\tilde{\iota}$ induced from the choices of the embeddings $\mathbb{Q}(\sqrt{\Delta})\hookrightarrow \mathbb{R}$
given by $\sqrt{\Delta}\mapsto \sqrt{\Delta}$ or $\sqrt{\Delta}\mapsto -\sqrt{\Delta}$.
For $\sigma \in {\rm Aut}(\mathbb{C})$,
set $\tilde {\iota}^\sigma (a)=(\tilde{\iota}(a))^\sigma$ $(a\in \mathbb{Q}(\sqrt{\Delta}))$.
For a triplet $(A,\Theta,\tilde{\iota})$,
we can consider $P^\sigma=(A^\sigma,\Theta^\sigma,\tilde{\iota}^\sigma)$
and  define the field of moduli $M'$ of $P=(A,\Theta,\tilde{\iota})$.
Namely, $P^\sigma$ is isomorphic to $P$ if and only if $\sigma \in {\rm Gal}(\mathbb{C}/M')$.
We note that the field of moduli $M'$ of $(A,\Theta,\tilde{\iota})$ is an algebraic extension of the field of moduli $M$ of $(A,\Theta)$
given in Definition \ref{Df1.2} (see \cite{S59} Proposition 8).

Suppose that $\Delta$ is equal to a sum of two squares of prime numbers.
Then, according to \cite{Geer} Chapter IX Section 2,
the moduli space of $(A,\Theta)$ is given by the Humbert surface,
which is isomorphic to $\langle PSL(2,\mathcal{O}_\Delta),\tau \rangle\backslash (\mathbb{H} \times \mathbb{H}) $,
and the forgetting mapping $(A,\Theta,\tilde{\iota})\mapsto (A,\Theta)$
corresponds to the action of the involution $\tau$ of $\mathbb{H} \times \mathbb{H}$ given by $(z_1,z_2)\mapsto (z_2,z_1)$.

From now on,
we consider the case $\Delta=1^2+2^2.$
By Proposition \ref{PropCF} (2),
the field of moduli $M $ of $(A,\Theta)$ coincides with $\mathbb{Q}(X,Y)=\mathbb{Q}(\frac{\mathfrak{B}}{\mathfrak{A}^3},\frac{\mathfrak{C}}{\mathfrak{A}^5})$ under our construction.
By virtue of Remark \ref{RemVV}  and the argument in the above paragraph,
the forgetting mapping $(A,\Theta,\tilde{\iota})\mapsto (A,\Theta)$ is described by  the double covering $\mathcal{R}:V\rightarrow \mathbb{P}(1:3:5)$ in 
Remark \ref{ThVV}.
On $(\mathfrak{A}:\mathfrak{B}:\mathfrak{C})\in \mathbb{P}(1:3:5)$,
there exist two points $(\mathfrak{A}:\pm \mathfrak{c}:\mathfrak{B}:\mathfrak{D})\in V$ corresponding to the choices of $\tilde{\iota}$.
So, the field of moduli $M'$ of $(A,\Theta,\tilde{\iota})$ is given by 
$\mathbb{Q}(X,\eta)$,
where 
$(X,\eta)=(\frac{\mathfrak{B}}{\mathfrak{A}^3},\frac{\mathfrak{D}}{\mathfrak{c}^3})$ is a system of affine coordinates of $V$.

Let $K$ be an imaginary quadratic extension of $K_0=\mathbb{Q}(\sqrt{5})$ and $(K,\{\varphi_j\})$ is a CM type.
We suppose that $(A,\Theta)$ is a principally polarized abelian surface of type $(K,\{\varphi_j\})$.
So, as in Section 1.1,
we have the isomorphism $\iota:K \simeq {\rm End}_0(A)$,
which induces $\tilde{\iota}=\iota|_{\mathbb{Q}(\sqrt{5})}: K\hookrightarrow {\rm End}_0(A)$.
In this case,
we are able to apply \cite{Shimura97} Chapter IV Proposition 1.
By this proposition, together with the definition of the field of moduli $M$ of $(A,\Theta)$,
we have 
$(A,\Theta,\iota)\simeq (A^\sigma.\Theta^\sigma,\iota^\sigma)$ for any $\sigma \in {\rm Gal}(\mathbb{C}/M)$.
Of course, this induces that $(A,\Theta,\tilde{\iota})\simeq (A^\sigma.\Theta^\sigma,\tilde{\iota}^\sigma)$.
Therefore, if $(A,\Theta,\tilde{\iota})$ is of CM type,
the field of moduli $M'$ of $(A,\Theta,\tilde{\iota})$ must coincide with the field of moduli $M$ of $(A,\Theta)$.
Hence, in our case,
 we have 
\begin{align}\label{Xeta}
\mathbb{Q}(X,Y) = \mathbb{Q} (X,\eta).
\end{align}

Here,
we suppose that there is a $K_0$-linear embedding $f: K\hookrightarrow M_2(\mathfrak{O}_B)$ with $f(\mathfrak{O}_K)\subset {\rm End}(A)$
and $(z_1^0,z_2^0)\in \mathbb{H} \times \mathbb{H}$ is the unique fixed point of $f(K)$.
By virtue of the argument \cite{S67} 2.7,
such an embedding $f $ determines a CM type $(K,\{\varphi_j\})$.
This CM type determines a pair $(A,\iota)$,
where $\iota:K \hookrightarrow {\rm End}_0(A)$.
Since   $(z_1^0,z_2^0)$ is the fixed point 
for $f(K)$, where $f$ satisfies  $f(\mathfrak{O}_K)\subset {\rm End}(A)$,
by the argument of \cite{Geer} Chapter IX Section 1,
we can see that
the abelian surface $A$ coming from the fixed point $(z_1^0,z_2^0)$
 has a principal polarization $\Theta$
and the pair $(A,\iota)$ is of CM type $(\mathfrak{O}_K,\{\varphi_1,\varphi_2\})$ in the sense of Definition \ref{DefPrin}.
So,
 by our construction via the Shioda-Inose structure of $\mathcal{F}=\{S(\mathfrak{A}:\mathfrak{B}:\mathfrak{C})\}$,
$\mathbb{Q}(X(z_1^0,z_2^0),Y(z_1^0,z_2^0))$ gives the field of moduli of the above $(A,\Theta)$.
Therefore,
by (\ref{Xeta}) and Theorem \ref{ThmCF},
we prove that 
$K^* (X(z_1^0,z_2^0),\eta(z_1^0,z_2^0))/K^*$ is the class field corresponding to the group $H_0$ of (\ref{H0}).
So, our $(V,\varphi)$ satisfies {\em (S-iii)}.
 \end{proof}

 \section{Cyclic quartic CM fields}

Generically,
our class field $K^*(X,Y)/K^*$ in Theorem \ref{ThmCF} is not the absolute class field of $K^*$,
because the group $I_{K^*} / H_0$  is not  equal to the ideal class group $I_{K^*}/P_{K^*}$.
Therefore, it is non-trivial to determine the structure of the  class field $K^*(X,Y)/K^*$.
Let us study the detailed structure of $K^*(X,Y)/K^*$.

In this section, we suppose that $K$ is normal over $\mathbb{Q}$.
Then, ${\rm Gal}(K/\mathbb{Q})$ is isomorphic to $(\mathbb{Z}/2\mathbb{Z})^2$ or $\mathbb{Z}/4\mathbb{Z}$
(see Example \ref{ExapS}).
However, if ${\rm Gal}(K/\mathbb{Q})\simeq (\mathbb{Z}/2\mathbb{Z})^2$,
then $K$ is not primitive in the sense of Definition \ref{DfPrimitive} and $K^*$ is an imaginary quadratic field.
So, our class field $K^*(X,Y)/K^*$ does not exceed the scope of Kronecker's Jugendtraum for imaginary quadratic fields.

Hence, in this section, we focus on the case of ${\rm Gal}(K/\mathbb{Q})\simeq \mathbb{Z}/4\mathbb{Z}$.
Then,   according to Example \ref{ExapS} (ii),
$K^*$ coincides with $K$.
Based on the results \cite{Shimura1}, \cite{HSW}, \cite{HW}, \cite{HHRWH} and \cite{HHRW} for cyclic extensions,
 we determine the structure of the Galois group $I_K/H_0$ of our class field $K(X,Y)/K$.

 \subsection{Integral basis of cyclic quartic  fields}

 Any   cyclic quartic  field over $\mathbb{Q}$ is given by
 \begin{eqnarray}\label{CQE}
 K=\mathbb{Q}\Big(\sqrt{A(\Delta + B \sqrt{\Delta})} \Big),
 \end{eqnarray}
 where
 $A,B,C,\Delta \in \mathbb{Z}$ and
 \begin{align}\label{CQEDeta}
 \begin{cases}
& A\equiv 1 \quad ({\rm mod}2),  \quad \text{squarefree}\\
& \Delta=B^2 +C^2, \quad (B>0, C>0), \quad \text{squarefree}\\
& {\rm GCD}(A,\Delta)=1.
 \end{cases}
 \end{align}
Let us consider the five cases (i),(ii),(iii),(iv) and (v) (see Table 3).

\begin{table}[h]
\center
\begin{tabular}{cc}
\toprule
Case  &   \\
 \midrule
 (i) &  $\Delta\equiv 0$ $({\rm mod} 2)$  \\ 
(ii) &  $\Delta\equiv B\equiv 1$ $({\rm mod} 2)$  \\
(iii) &     $\Delta\equiv 1$ $({\rm mod} 2)$ , $B\equiv 0$ $({\rm mod} 2)$, $A+B \equiv 3$ $({\rm mod} 4)$\\
(iv) &     $\Delta\equiv 1$ $({\rm mod} 2)$ , $B\equiv 0$ $({\rm mod} 2)$, $A+B \equiv 1$ $({\rm mod} 4)$, $A\equiv C$ $({\rm mod} 4) $\\
(v) &     $\Delta\equiv 1$ $({\rm mod} 2)$ , $B\equiv 0$ $({\rm mod} 2)$, $A+B \equiv 1$ $({\rm mod} 4)$, $A\equiv -C$ $({\rm mod} 4) $\\
 \bottomrule
\end{tabular}
\caption{The five cases (i),(ii),(iii),(iv) and (v)}
\end{table}
 
 \begin{prop} (\cite{HSW}, \cite{HW}, \cite{HHRWH} or \cite{HHRW})
 \label{HardyProp}
Let $K$ be a cyclic quartic field of (\ref{CQE}). 

(1)  The conductor of the field $K$ is given by
$
2^l \Delta |A|,
$
 where
 \begin{align*}
l=
 \begin{cases}
 & 3 \quad ( \text{ case  (i)  and  (ii) } ), \\
 &2 \quad ( \text{ case  (iii) }), \\
 & 0 \quad( \text{ case  (iv)  and  (v) } ).
 \end{cases}
 \end{align*}
 
 (2) The discriminant of the field $K$ is given by
$
2^e \Delta^3 A^2,
$
where
 \begin{align*}
 e
 =
 \begin{cases}
 & 8 \quad ( \text{case (i) }),  \\
 & 6 \quad ( \text{case (ii) })\\
 &  4 \quad ( \text{case (iii) }),\\
 & 0 \quad ( \text{case (iv) and (v) }).
 \end{cases}
 \end{align*}

(3)
Set
$
\alpha=\sqrt{A(\Delta + B \sqrt{\Delta})},  \beta=\sqrt{A(\Delta - B \sqrt{\Delta})}.
$
A system of basis of the ring of integers $\mathfrak{O}_K$ of $K$ is given by  Table 4.
\begin{table}[h]
\center
\begin{tabular}{cc}
\toprule
Case  &  Basis \\
 \midrule
 (i) \vspace{2mm}&  $1, \sqrt{\Delta},\alpha,\beta $  \\ 
(ii) \vspace{2mm} &  $\displaystyle  1, \frac{1+\sqrt{\Delta}}{2},\alpha,\beta $  \\
(iii)\vspace{2mm} &     $\displaystyle  1,  \frac{1+\sqrt{\Delta}}{2},\frac{\alpha+\beta}{2},\frac{\alpha-\beta}{2} $\\
(iv) \vspace{2mm} &     $\displaystyle  1,  \frac{1+\sqrt{\Delta}}{2},\frac{1+\sqrt{\Delta}+\alpha+\beta}{4},\frac{1-\sqrt{\Delta}+\alpha-\beta}{4} $\\
(v) &     $\displaystyle  1,  \frac{1+\sqrt{\Delta}}{2},\frac{1+\sqrt{\Delta}+\alpha-\beta}{4},\frac{1-\sqrt{\Delta}+\alpha+\beta}{4}$\\
 \bottomrule
\end{tabular}
\caption{The basis of the ring of integers for the cases (i),(ii),(iii),(iv) and (v)}
\end{table}

  \end{prop}

\subsection{Complex tori from  cyclic quartic CM fields}

 A cyclic quartic field $K$ of (\ref{CQE}) gives a primitive CM type (see Example \ref{ExapS}), if $A<0$.
We call such fields  cyclic quartic CM fields.

In Section 3.1,
we obtained a system of basis $\alpha_1,\cdots,\alpha_4$ of the ring $\mathfrak{O}_K$ of integers of  $K$.
As Section 1.1,
we obtain the lattice $\Lambda(\mathfrak{O}_K)$ of $\mathbb{C}^2$ from $\alpha_1,\cdots,\alpha_4$.
Then, we have a complex torus $\mathbb{C}^2/\Lambda(\mathfrak{O}_K)$.
If the Riemann form $E$ of (\ref{RiemannShimura}) on $\Lambda(\mathfrak{O}_K)\times \Lambda(\mathfrak{O}_K)$ is $\mathbb{Z}$-valued,
the complex torus $\mathbb{C}^2/\Lambda(\mathfrak{O}_K)$ becomes a abelian surface of CM  type $(\mathfrak{O}_K,\{id,\varphi\})$
(see Definition \ref{DefPrin} and Proposition \ref{PropPrincipal}).

Recall that the  Riemann form of (\ref{RiemannShimura})  depends on  $\zeta \in K$
 such that
 $K=K_0(\zeta)$, $-\zeta^2 \in K_0$ is totally positive and ${\rm Im}(\zeta^{\varphi_j})>0$.
 In our study for the field of (\ref{CQE}),
 we canonically take
 \begin{align}\label{ourzeta}
 \zeta=\frac{\sqrt{A(\Delta + B\sqrt{\Delta})}}{\kappa}, \quad \quad (\kappa \in \mathbb{Q}).
 \end{align}

 \begin{prop}
 For a cyclic field $K$ of (\ref{CQE}) and a number $\zeta $ of (\ref{ourzeta}),
 the  matrix $(E(\alpha_j,\alpha_k))_{j,k=1,\cdots,4}$ coming from the alternating Riemann form $E$ of (\ref{RiemannShimura}) is given as Table 5.
 \begin{table}[h]
\center
\begin{tabular}{cc}
\toprule
Case  &   Matrix $(E(\alpha_j,\alpha_k))_{j,k=1,\cdots,4}$ \\
 \midrule
 (i) \vspace{2mm}&  $\displaystyle   - \frac{4\Delta A}{\kappa} \begin{pmatrix} 0 & 0& 1 & 0 \\ 0 & 0 & B & C \\ -1 & -B & 0 & 0 \\ 0 & -C &0 &0 \end{pmatrix}(=:M_1)$  \\ 
(ii) \vspace{2mm} &  $\displaystyle    - \frac{2\Delta A}{\kappa} \begin{pmatrix} 0 & 0& 2 & 0 \\ 0 & 0 & (1+B)  & C \\ -2 & -(1+B) & 0 & 0 \\ 0 & -C &0 &0 \end{pmatrix} (=:M_2)$  \\
(iii)\vspace{2mm} &     $\displaystyle   - \frac{\Delta A}{\kappa} \begin{pmatrix} 0 & 0& 2 & 2 \\ 0 & 0 & (1+B+C)  & (1+B-C) \\ -2 & -(1+B+C) & 0 & 0 \\ -2 & -(1+B-C) &0 &0 \end{pmatrix}(=:M_3)$ \\
(iv) \vspace{2mm} &     $\displaystyle   - \frac{\Delta A}{\kappa} \begin{pmatrix} 0 & 0& 1 & 1 \\ 0 & 0 & (1+B+C)/2  & (1+B-C)/2 \\ -1 & -(1+B+C)/2 & 0 & B/2 \\ -1 & -(1+B-C)/2 &-B/2 &0 \end{pmatrix} (=:M_4)$\\
(v) &    $\displaystyle   - \frac{\Delta A}{\kappa} \begin{pmatrix} 0 & 0& 1 & 1 \\ 0 & 0 & (1+B-C)/2  & (1+B+C)/2 \\ -1 & -(1+B-C)/2 & 0 & B/2 \\ -1 & -(1+B+C)/2 &-B/2 &0 \end{pmatrix}(=:M_5) $\\
 \bottomrule
\end{tabular}
\caption{The  matrices for the Riemann form $E$ of (\ref{RiemannShimura}) on $\Lambda(\mathfrak{O}_K) \times \Lambda(\mathfrak{O}_K)$}
\end{table}

 \end{prop}

 \begin{proof}
 Using Proposition \ref{HardyProp} (3),
 we can prove this by a straightforward calculation.
 \end{proof}

 \subsection{Detailed properties of ideal class groups and Galois groups}

 The Galois group  for our class field is given by $I_K/H_0$ (see (\ref{H0})).
 However,  the structure of $H_0$  is complicated.
 To study $H_0$, we need some detailed properties of CM fields.

 Set 
 \begin{align}\label{K/K0}
  I(K/K_0)=\{\mathfrak{a}\in I_K | \mathfrak{a} (\mathfrak{a}^{\rho})^{-1} \in P_{K} \},\quad I_0(K/K_0)=P_K \{\mathfrak{a}\in I_K|\mathfrak{a} ^\rho = \mathfrak{a} \}.
 \end{align}

 Let $P_{K_0}^+$ be the subgroup of $P_{K_0}$ consisting of all principal ideals generated by totally positive elements.
 Let 
$E_{K_0}$
be the group  of all units in $\mathfrak{O}_{K_0}$.
Let $E_{K_0}^+$ be the group of  all totally positive units  in $\mathfrak{O}_{K_0}$. 
Set $E_{K_0}^2=\{u^2 | u \in E_{K_0}\}$.
Then, there exists  $\varepsilon \in \mathbb{Z}_{\geq 0}$ such that
\begin{align}\label{EUnit}
2^{\varepsilon} =[E_{K_0}^+:E_{K_0}^2] =[P_{K_0} :P_{K_0}^+].
\end{align}

We have the mapping 
\begin{align}\label{NormMap}
N_{K/K_0}:I_K \rightarrow I_{K_0}.
\end{align}
There exists $\delta \in \mathbb{Z}_{\geq 0}$ such that
\begin{align}\label{delta}
2^\delta =[I_{K_0}: P_{K_0}^+ N_{K/K_0}(I_K)].
\end{align}

 Set
 $
 R_K=\{\mathfrak{a}\in I_K | N_{K/K_0} (\mathfrak{a}) \in P_{K_0}\}
 $ 
 and 
$R_K^+ =\{\mathfrak{a}\in I_K | N_{K/K_0} (\mathfrak{a}) \in P_{K_0}^+\}.$
The mapping $N_{K/K_0}$ of (\ref{NormMap}) induces an embedding
\begin{align}\label{IREmb}
I_K/R_K \hookrightarrow I_{K_0}/P_{K_0}.
\end{align}
Then, we have
\begin{align}\label{RR+}
[R_K : R_K^+]=2^{\varepsilon -\delta}.
\end{align}

 \begin{lem} \label{ShimuLem}(\cite{Shimura1} Appendix)
(1) There exist $\eta,\beta,\gamma \in\mathbb{Z}_{\geq 0}$ such that
\begin{align}\label{etabetagamma}
 2^\eta = [I_{K_0} \cap P_K : P_{K_0}],\quad
2^\beta = [I(K/K_0): I_0(K/K_0)] ,\quad
2^\gamma = [I_0(K/K_0) : I_{K_0} P_K].
\end{align}

(2) The relations
\begin{align}\label{beta.E}
 [I_K:I_{K_0}P_K] = 2^\eta \frac{h_{K}}{h_{K_0}}, \quad
 2^{\eta-\beta-\gamma} \frac{h_K}{h_{K_0}} = [I_K: I(K/K_0)],
 \quad
 \beta \leq \varepsilon, \quad
 2^\gamma = 2^{\eta + t-1} [N_{K/K_0} (E_K) : E_{K_0}^2 ].
\end{align}
hold.
Here,
$E_K$ be the group consists of all units in $\mathfrak{O}_K $ and 
$t$ is the number of prime ideals of $K_0$ ramified in $K$.

(3)
If  $2^{-\eta} h_{K_0}$ is an odd number, 
then 
\begin{align}\label{2Tor}
I(K/K_0)=\{\mathfrak{a}\in I_K|  \mathfrak{a}^2 \in I_{K_0}P_K \}.
\end{align}
Moreover, $2^{\beta+\gamma}$ is the number of elements of order $1$ or $2$ in the ideal class group $I_K/P_K $.
 \end{lem}

 Let $(K,\{\varphi_j\})$ be a CM type and  $(K^*,\{\psi_k\})$ be  its reflex.
 In fact,  the mapping $\Phi^*$ of (\ref{Phi*}) gives a mapping $I_{K^*} \rightarrow R_{K}^+$. 
 Set
 \begin{align}\label{I(Phi*)}
 I(\Phi^*)= \Big\{\mathfrak{a} \in I_{K^*}  | \mathfrak{a}^{\Phi^*} \in P_K \Big\}.
 \end{align}

 \begin{lem}(\cite{Shimura1} Appendix) \label{ShimuUme}
 There exists an embedding
 $
 I(\Phi^*) / H_0 \hookrightarrow E_{K_0}^+ / N_{K/K_0} (E_K).
 $
 \end{lem}

 \begin{lem} \label{SimuQ} (\cite{Shimura1} Appendix)
 If $K$ is  an cyclic extension over $\mathbb{Q} $,
 then, $K^*=K$.
 Moreover, it holds that
 $
  H_0 \subset I_0 (K/K_0) ,  I(\Phi^*)\subset I(K/K_0) 
  $
 and
\begin{align}\label{hougan1}
 I_{K_0} P_{K} \subset H_0 \subset I(K/K_0) \subset I_K .
 \end{align}
  \end{lem}

 For proofs of the above arguments, 
 see \cite{Shimura1}.
   
 \begin{lem} \label{OurShimu}
 Let $K_0$ be a totally real number field and $K$ be an imaginary quadratic extension of $K_0$. 
 Let us consider the following three conditions:
 \begin{align}\label{OurCondition}
 \begin{cases}
& {\rm (C1)} : \quad h_{K_0}=1, \\
&{\rm (C2)} :\quad E_{K_0}^+=E_{K_0}^2 , \\
&{\rm (C3)}:  \quad K / \mathbb{Q} \text{ is cyclic.}
 \end{cases}
 \end{align} 

 (1) 
 If $K$ satisfies the condition {\rm (C1)},
then
$ \eta=0$ and $I_K=R_K$ hold.
 
 (2)
 If $K$ satisfies the conditions {\rm (C1)} and {\rm (C2)},
 then
$\varepsilon =0, \beta =0, I_K=R_K^+$ and $I(\Phi^*) = H_0$ hold. 
 
(3) Suppose the ideal class group $I_K/P_K$ is isomorphic to  
\begin{align}\label{ICGp}
 (\mathbb{Z}/2\mathbb{Z})^{r_1}\oplus (\mathbb{Z}/2^2\mathbb{Z})^{r_2}\oplus \cdots  \oplus (\mathbb{Z}/2^k\mathbb{Z})^{r_k} \oplus G_1,
 \end{align}
where $G_1$ does not contain any elements of order $2m$ ($m\in \mathbb{Z}$).
Set $\sum_{j=1}^k r_j =r.$
 If $K$ satisfies the conditions {\rm (C1)}, {\rm (C2)} and {\rm (C3)},
then $\gamma=r$, $\displaystyle [I(K/K_0):P_K]=2^r$ 
hold.
  \end{lem}
 
 \begin{proof}
 (1) Assuming (C1), we have $\eta=0$ and $I_K=R_K$
from    (\ref{IREmb})
 and  (\ref{etabetagamma}).

(2) By (C2) and (\ref{EUnit}),  we obtain $\varepsilon=0$. 
From (\ref{beta.E}), $\beta=0$ holds.
Moreover, (\ref{delta}), (\ref{RR+}) and the above (1) assure $I_K=R_K^+$.
Note that $[E_{K_0}^+:E_{K_0}^2]=[E_{K_0}^+:N_{K/K_0}(E_K)][N_{K/K_0}(E_K):E_{K_0}^2]$.
 Together with Lemma \ref{ShimuUme}, we have $I(\Phi^*) = H_0$.

(3) 
If $I_K/P_K$ is isomorphic to the group of (\ref{ICGp}), 
the number of elements of order $1$ or $2$ of $I_K/P_K$ is equal to $2^r$.
By (C1) and Lemma \ref{ShimuLem} (3),  we have $\beta+ \gamma =r$. 
Together with the above (2), we have  $\gamma=r$.
So, by  (\ref{etabetagamma}), we obtain
$[I(K/K_0):P_K]=2^r$.
\end{proof}

Now, we determine the structure of the group $H_0$ when $K$ satisfies the conditions (C1), (C2) and (C3). 

\begin{thm}\label{ThmHantei}
Let $K$ be a quartic CM field.
Suppose the ideal class group is isomorphic to the group of (\ref{ICGp}).
If $K$ satisfies the conditions
 {\rm (C1),(C2)} and {\rm (C3)} in Lemma \ref{OurShimu},
 then
 \begin{align}\label{ClassStrH0}
 I_K/H_0 \simeq (\mathbb{Z}/2\mathbb{Z})^{r_2} \oplus  \cdots \oplus (\mathbb{Z}/2^{k-1}\mathbb{Z})^{r_k} \oplus G_1
 \end{align}
holds.
 \end{thm}

\begin{proof}

Since $K/\mathbb{Q}$ is  cyclic,
we have 
\begin{eqnarray}\label{GalC}
 {\rm Gal} (K / \mathbb{Q}) =\{  id, \sigma, \sigma^2=\rho, \sigma ^3 \},
\end{eqnarray}
where $(K,\{id,\sigma\})$ gives a CM type of $K$ (see Example \ref{ExapS}).
Recall that  $(K,\{id,\sigma\})$ is primitive. 
The corresponding reflex is given by $(K,\{id, \sigma^{-1}\})$.
The mapping $\Phi^*$ of (\ref{Phi*}) has the form 
\begin{eqnarray}\label{Phi4}
\mathfrak{a} \mapsto \mathfrak{a}^{\Phi^*}=\mathfrak{a}\mathfrak{a}^{\sigma^{-1}}.
\end{eqnarray}

Consider the canonical projection
$
p: I_K \rightarrow I_K/P_K 
$
given by
$\mathfrak{a} \mapsto [\mathfrak{a}]$. 
We set $p(I(K/K_0))=\overline{I(K/K_0)}$ and  $p(H_0)=\overline{H_0}$.
They are subgroups of the ideal class group $I_K/P_K$.
We shall prove $\overline{H_0}=\overline{I(K/K_0)}.$

From Lemma \ref{ShimuLem} (3) (especially (\ref{2Tor})) and Lemma \ref{OurShimu} (1), 
any $[\mathfrak{b}]\in \overline{I(K/K_0)}$ is  a $2$-torsion element. 
Especially, we have
\begin{eqnarray}\label{Ib2}
[\mathfrak{b}]^{-1}=[\mathfrak{b}].
\end{eqnarray}
Due to Lemma \ref{OurShimu} (3),
for some $[\mathfrak{b}_1],\cdots, [\mathfrak{b}_r] \in \overline{I(K/K_0)},$ 
\begin{align}\label{2Torsion}
\overline{I(K/K_0)}= \{[\mathfrak{b}_1]^{e_1} \cdots [\mathfrak{b}_r]^{e_r}| e_1,\cdots, e_r \in \{0,1\}\}\simeq (\mathbb{Z}/2\mathbb{Z})^r ,
\end{align}
where $r=\sum_{j=1}^k r_j$.
Here,
each factor $(\mathbb{Z}/2\mathbb{Z})$ of (\ref{2Torsion}) is a subgroup of some $(\mathbb{Z}/2^{j} \mathbb{Z})$ in (\ref{ICGp}).
Moreover,
according to the definition of $\overline{I(K/K_0)}$ of (\ref{K/K0}),
we have
\begin{align}\label{brho}
[\mathfrak{b}]^\rho = [\mathfrak{b}] \quad\quad ([\mathfrak{b}]\in\overline{I(K/K_0)}).
\end{align}

On the other hand,
by (\ref{I(Phi*)}), Lemma \ref{OurShimu} (2) and  (\ref{Phi4}),  
we have ${\rm Ker}( \Phi)=\overline{H}_0$
and
\begin{eqnarray}\label{H0C}
[\mathfrak{a}]\in \overline{H_0} \Longleftrightarrow [\mathfrak{a}]^\sigma = [\mathfrak{a}]^{-1}.
\end{eqnarray}

Under the condition (C1),  the sequence (\ref{hougan1}) becomes to be 
$
 P_{K} \subset H_0 \subset I(K/K_0) \subset I_K.
$
Let us assume
\begin{align}\label{Assume}
\overline{H_0} \subsetneqq \overline{I(K/K_0)}.
\end{align}
Then, 
without loss of generality,
we can take 
$[\mathfrak{b}_1],\cdots, [\mathfrak{b}_r]$ in (\ref{2Torsion})
and
$s\in \mathbb{Z}$ $(0 < s\leq r)$  such that 
$$
\overline{H_0}=\{[\mathfrak{b}_{s+1}]^{e_{s+1}}\cdots [\mathfrak{b}_{r}]^{e_r}| e_{s+1},\cdots, e_r \in \{0,1\}\}.
 $$
So, 
according to  (\ref{2Torsion}),
$J_0=\overline{I(K/K_0)}/\overline{H_0}$ must be in the form
\begin{align}\label{J0G}
J_0= \{[\mathfrak{b}_{1}]^{e_{1}}\cdots [\mathfrak{b}_{s}]^{e_s}| e_{1},\cdots, e_s \in \{0,1\}\}.
\end{align}
For any $[\mathfrak{b}]\in \overline{I(K/K_0)}$,
we have
 $[\mathfrak{b}]^\sigma \in \overline{I(K/K_0)}$
since
$([\mathfrak{b}]^\sigma)^\rho=([\mathfrak{b}]^\rho)^\sigma=[\mathfrak{b}]^\sigma.$
By virtue of (\ref{H0C}),
we have
\begin{align}\label{bjk}
[\mathfrak{b}_j]^\sigma \in J_0,\quad  [\mathfrak{b}_j]^\sigma \not= [\mathfrak{b}_j] \quad \quad ( j \in\{1,\cdots, s\}).
\end{align}
For $j\in\{1,\cdots,s\},$ set
\begin{align}\label{cj}
[\mathfrak{c}_j]=[\mathfrak{b}_j][\mathfrak{b}_j]^\sigma (\not= id_{\overline{I(K/K_0)}}).
\end{align}
According to (\ref{J0G}) and (\ref{bjk}),
$[\mathfrak{c}_j] \in J_0$.
From (\ref{GalC}) and (\ref{brho}),
we have
$$
[\mathfrak{c}_j]^\sigma =([\mathfrak{b}_j][\mathfrak{b}_j]^\sigma)^\sigma=[\mathfrak{b}_j]^\sigma [\mathfrak{b}_j]^\rho =[\mathfrak{b}_j]^\sigma [\mathfrak{b}_j]=[\mathfrak{c}_j].
$$
Therefore,
$[\mathfrak{c}_j]\in \overline{H_0}$.
This is a contradiction.

Hence, the assumption (\ref{Assume}) is not true.
Therefore, we have $\overline{I(K/K_0)}=\overline{H_0}$. 
According to (\ref{ICGp}) and (\ref{2Torsion}),  we have (\ref{ClassStrH0}).
\end{proof}

\begin{cor}
Let $K$ be a quartic CM field as in Theorem \ref{ThmHantei}.
Set $r=\sum_{j=1}^k r_j.$
The unramified class field $K(X,Y)/K$ corresponding to  the group $H_0$ of (\ref{H0}) is an extension over $K$ of degree $\displaystyle \frac{h_K}{2^r}$.
\end{cor}

 \begin{proof}
 It is clear because we have Theorem \ref{ThmClassField} and Theorem \ref{ThmHantei}.
 \end{proof}

 \subsection{Cyclic quartic CM fields over $\mathbb{Q}(\sqrt{5})$}

 Letting $K_0$ be a real quadratic field,
 the group $E_{K_0}$  of units  in $\mathfrak{O}_{K_0}$ is 
 given by the direct product  of
 $\{id,-id\}$
 and $\{\varepsilon _0 ^n |n\in\mathbb{Z}\}$.
 Here, $\varepsilon _0$ is called a fundamental unit of $K_0$.
 If $N_{K/K_0} (\varepsilon_0)=-1,$ 
 then the group $E_{K_0}^+$ of totally positive units in $\mathfrak{O}_{K_0}$ is given by $\{\varepsilon_0^{2n}|n\in\mathbb{Z}\}$ and
 we have
 $E_{K_0}^+=E_{K_0}^2$.

 In this section, we study the  case for $K_0=\mathbb{Q}(\sqrt{5})$ in detail.
 The class number of $\mathbb{Q}(\sqrt{5})$ is equal to $1$.
 A fundamental unit of $\mathbb{Q}(\sqrt{5})$ is given by $\displaystyle \frac{1-\sqrt{5}}{2}$,
 that satisfies $N_{\mathbb{Q}(\sqrt{5})/\mathbb{Q}} \Big(\displaystyle \frac{1-\sqrt{5}}{2}\Big)=-1.$
 Therefore, 
 if $K$ is an imaginary quadratic extension over $\mathbb{Q}(\sqrt{5})$,
 then $K$ satisfies the conditions (C1) and (C2) in Lemma \ref{OurShimu}.
Moreover, if $K/\mathbb{Q}$ is cyclic and $K/\mathbb{Q}(\sqrt{5})$ is totally imaginary quadratic,
 it follows that
 \begin{align}\label{bc5}
 (B,C)=(1,2) \text{ or } (2,1)
 \end{align}
 (see (\ref{CQEDeta})).

\begin{thm} \label{ThmC5}
Let $K$ of (\ref{CQE}) be a totally imaginary extension of $\mathbb{Q}(\sqrt{5})$ such that  $K/\mathbb{Q}$ is cyclic.
Set $\alpha=\sqrt{A(5+B\sqrt{5})}$.
The Riemann form $E$ of (\ref{RiemannShimura}) given by $\zeta$ in Table 6
gives a principal polarization  on the abelian variety $A(\mathfrak{O}_K)$.

\begin{table}[h]
\center
\begin{tabular}{cc}
\toprule
Case  &  $\zeta$ of (\ref{ourzeta}) \\
 \midrule
(ii) \vspace{2mm} &  $\displaystyle\frac{\alpha}{-4\Delta A} = \frac{\alpha}{-20 A} $  \\
(iii)\vspace{2mm} &  $\displaystyle\frac{\alpha}{-2\Delta A} =  \frac{\alpha}{-10 A}$  \\
(iv) \vspace{2mm} &  $\displaystyle\frac{\alpha}{-\Delta A} =   \frac{\alpha}{-5 A}$    \\
(v) &      $\displaystyle\frac{\alpha}{-\Delta A} =   \frac{\alpha}{-5 A}$     \\
 \bottomrule
\end{tabular}
\caption{The number $\zeta$ which gives a principal polarization for $\Delta=5$}
\end{table}
\end{thm}

 \begin{proof}
If the alternating Riemann form $E$ of (\ref{RiemannShimura}) is $\mathbb{Z}$-valued, 
the determinant of the matrix $M_j$ in Table 5 is a perfect square number. 
The positive root of this determinant is called the Pfaffian of $E$.
If the Pfaffian of $E$ is equal to $1$,
then the Riemann form $E$ induces a principal polarization
on the complex torus $\mathbb{C}^2/\Lambda(\mathfrak{O}_K)=A(\mathfrak{O}_K)$.
Hence, it is sufficient to see that
the number $\zeta$ in Table 6
gives the $\mathbb{Z}$-valued matrix $M_j$ in Table 5
and the determinant of $M_j$ is equal to $1$.

(ii) In this case, we consider the matrix $M_2$  in Table 5.
 Putting $(B,C)=(1,2) $ and $\kappa=-4\Delta A$,  we have
 $$
 M_2=  \begin{pmatrix} 0 & 0& 1 & 0 \\ 0 & 0 & 1  & 1 \\ -1 & -1 & 0 & 0 \\ 0 & -1 &0 &0 \end{pmatrix}.
 $$
 
 (iii) In this case, we consider the matrix $M_3$  in Table 5.
 Putting $(B,C)=(2,1) $ and $\kappa=-2\Delta A$,  we have
 $$
 M_3=  \begin{pmatrix} 0 & 0& 1 & 1 \\ 0 & 0 & 2  & 1 \\ -1 & -2 & 0 & 0 \\ -1 & -1 &0 &0 \end{pmatrix}.
 $$

 (iv) In this case, we consider the matrix $M_4$  in Table 5.
 If we consider the case (iv),  putting $(B,C)=(1,2) $ and $\kappa=- \Delta A$,  we have
 $$
 M_4=  \begin{pmatrix} 0 & 0& 1 & 1 \\ 0 & 0 & 2  & 1 \\ -1 & -2 & 0 & 1 \\ -1 & -1 &-1 &0 \end{pmatrix}.
 $$
 
 (v)
 In this case, we consider the matrix $M_5$  in Table 5.
 If we consider the case (v),  putting $(B,C)=(1,2) $ and $\kappa=- \Delta A$, ,  we have
 $$
 M_5=  \begin{pmatrix} 0 & 0& 1 & 1 \\ 0 & 0 & 1  & 2 \\ -1 & -1 & 0 & 1 \\ -1 & -2 &-1 &0 \end{pmatrix}.
 $$
 
 \end{proof}

 Theorem \ref{ThmC5} implies that
 we only need principally polarized abelian surfaces to construct unramified class fields $k_0$ in Section 1.4 for $\Delta=5$.
 Hence, Theorem \ref{ThmCF}  is available for every cyclic quartic CM field $K$ for $\Delta=5$. 
 
 \begin{rem}
 The construction of class fields in Section 1.4 for $\Delta=5$ is simpler than the cases of $\Delta>5$.
 In general,  in those cases
 we need  non-principally polarized abelian surfaces.  
 \end{rem}

 \subsection{Examples  }
 
 In this section,
 we  apply our results
  of the icosahedral invariants $X$ and $Y$ 
  to concrete cyclic CM fields $K$.
  We construct unramified class fields $k_0$ in Section 1.4
   over cyclic quartic CM fields $K$ with the maximal real subfield $\mathbb{Q}(\sqrt{5})$.
  According to Theorem \ref{ThmCF} and Theorem \ref{ThmC5}, 
  such class fields are given by the special values of $X$ and $Y$ of (\ref{XY}). 

 Hardy, Hudson, Richmann, Williams and Hiltz \cite{HHRWH}
 listed imaginary cyclic quartic fields $K$ with small conductor.  
Using their result, we will obtain examples of our construction of class fields.
 In Section 3.5.1, we shall give an example such that $[k_0:K]=1$.
 In Section 3.5.2, we shall give an example such that $[k_0:K]=5>1.$
 
 \subsubsection{Case  $K=\mathbb{Q}\left(\sqrt{-(5+\sqrt{5})}\right)$}

 Let us take $K=\mathbb{Q}\left(\sqrt{-(5+\sqrt{5})}\right)$.
 From \cite{HHRWH}, the conductor of $K$ is $5$.
We consider the case (ii) in Table 3.

 According to Theorem \ref{ThmC5},
taking
 $
 \zeta=\frac{ \sqrt{- (5+\sqrt{5})}}{20},
 $
 we have the Riemann form $E$ of (\ref{RiemannShimura}).  
Next, we take a system $\{\alpha_1,\cdots, \alpha_4\}$ of  basis as in Table 4.
Then, we have the lattice
 $
 \Lambda=\langle u(\alpha_1),\cdots , u(\alpha_4) \rangle
 $
 of $\mathbb{C}^2,$
where
\begin{align*}
\begin{cases}
& u(\alpha_1)=\begin{pmatrix} 1 \\ 1 \end{pmatrix},  \quad  u(\alpha_2)=\begin{pmatrix} \frac{1+\sqrt{5}}{2} \\ \frac{1-\sqrt{5}}{2} \end{pmatrix},\\
& u(\alpha_3)= \begin{pmatrix} \sqrt{-(5+\sqrt{5})}  \\  \sqrt{-(5-\sqrt{5})}  \end{pmatrix}, \quad  u(\alpha_4)=\begin{pmatrix} \sqrt{-(5-\sqrt{5})} \\  -\sqrt{-(5+\sqrt{5})}  \end{pmatrix}.
\end{cases}
\end{align*}

Putting
 $
 \lambda_1=u(\alpha_1),  \lambda_2=u(\alpha_2),  \lambda_3=u(\alpha_3)-u(\alpha_4)$ and 
  $\lambda_4=u(\alpha_4),$
 we have
 \begin{align*}
 E(\lambda_j,\lambda_k)=
 \begin{pmatrix}
 0 & 0 & 1 & 0 \\
 0 & 0 &  0 & 1 \\
 -1 & 0 & 0 &0 \\
 0 & -1 & 0 &0
  \end{pmatrix}.
 \end{align*}
We set
 $
(M_1M_2)=(\lambda_1\lambda_2\lambda_3\lambda_4),
$
where
$M_1,M_2 \in M(2,\mathbb{C}).$
So, we have
$
\Omega=-M_2^{-1} M_1 \in \mathfrak{S}_2.
$

 Set 
$
\begin{pmatrix} A _0& B_0 \\ C_0 &D_0  \end{pmatrix} =
\begin{pmatrix}
0 & 1 & 0 & 0\\
1 &0 & 0 & 0\\
0 & 0 &0 & 1\\
0 & 0  & 1 &0
  \end{pmatrix}
  \in Sp(2,\mathbb{Z}).
$

Then, we have
$
 \tilde{\Omega}=(A_0 \Omega + B_0)(C_0 \Omega + D_0)^{-1}
=
\begin{pmatrix}
\tau_1 & \tau_2 \\
\tau_2 & \tau_3
\end{pmatrix},
$
where
$
\tau_1= \frac{1}{5}\sqrt{-1}\sqrt{\frac{25}{2}+\sqrt{5}},
\tau_2=\frac{1}{5} \sqrt{-1}\sqrt{5-\sqrt{5}}$
and
$
 \tau_3= \frac{1}{5}\sqrt{-1}\sqrt{\frac{5}{2}+\sqrt{5}}. 
$
 This satisfies
$
\tau_1  - \tau_2 -\tau_3=0.
$ 
 So, from $\tau_2$ and $\tau_3$,
we have $(z_1^0,z_2^0)\in\mathbb{H}\times \mathbb{H}$ as in (\ref{z10,z20}).
According to Theorem  \ref{ThmCF}, 
using the pair of modular functions $(X,Y)$ of (\ref{XY}),
the unramified class field $k_0$ over $K=\mathbb{Q}\left(\sqrt{-(5+\sqrt{5})}\right)$ corresponding to the group $H_0$ of (\ref{H0}) is given by
$
k_0=K(X(z_1^0,z_2^0),Y(z_1^0,z_2^0)).
$

Due to \cite{HHRWH},
the class number $h_K$ of $K=\mathbb{Q}\left(\sqrt{-(5+\sqrt{5})}\right)$ is equal to $2$.
Therefore, 
$I_K/P_K$ is isomorphic to $(\mathbb{Z}/2\mathbb{Z}).$
Applying Theorem \ref{ThmHantei} to this case,
we have
$[k_0:K]=1.$
This case does not give a non-trivial class-field  over the CM field $K$.

 \begin{rem}
In fact,
Murabayashi and Umegaki \cite{MU} proved that the field of moduli of principally polarized abelian surface corresponding to $K=\mathbb{Q}\left(\sqrt{-(5+\sqrt{5})}\right)$ coincides with $\mathbb{Q}$.
So, our result does not contradict  their result.
 \end{rem}

 \begin{rem}
 In this case,
 since $h_K=2$,
 there exists the absolute class field $C(K)$ over $K$
 such that $[C(K):K]=2.$
 In \cite{NaganoShigaClassField} Example 5.2,
we have an explicit construction of
$C(K)/K$ using another explicit modular function 
derived from the moduli of a PEL-structure for $8$-dimensional  abelian varieties.
 \end{rem}

 \subsubsection{Case  $K=\mathbb{Q}\left(\sqrt{-37(5+2\sqrt{5})}\right)$ }

Let us take $K=\mathbb{Q}\left(\sqrt{-37(5+2\sqrt{5})}\right)$.
From \cite{HHRWH}, the conductor of $K$ is $185$.
We consider the case (v) in Table 3.

According to Theorem \ref{ThmC5},
taking
 $
 \zeta=\frac{ \sqrt{-37 (5+2\sqrt{5})}}{185},
 $
 we have the Riemann form $E$ of (\ref{RiemannShimura}).  
Next, we take a system $\{\alpha_1,\cdots, \alpha_4\}$ of  basis as in Table 4.
Then, we have the lattice
 $
 \Lambda=\langle u(\alpha_1),\cdots , u(\alpha_4) \rangle
 $
 of $\mathbb{C}^2,$
 where
 \begin{align*}
\begin{cases}
& u(\alpha_1)=\begin{pmatrix} 1 \\ 1 \end{pmatrix},  \quad  u(\alpha_2)=\begin{pmatrix} \frac{1+\sqrt{5}}{2} \\ \frac{1-\sqrt{5}}{2} \end{pmatrix},\\
& u(\alpha_3)=\frac{1}{4} \begin{pmatrix} 1+\sqrt{5}+\sqrt{-37(5+2\sqrt{5})} -\sqrt{-37(5-2\sqrt{5})} \\ 1-\sqrt{5}+\sqrt{-37(5-2\sqrt{5})} +\sqrt{-37(5+2\sqrt{5})}  \end{pmatrix}, \\
& u(\alpha_4)=\frac{1}{4} \begin{pmatrix} 1-\sqrt{5}+\sqrt{-37(5+2\sqrt{5})} +\sqrt{-37(5-2\sqrt{5})} \\ 1+\sqrt{5}+\sqrt{-37(5-2\sqrt{5})} -\sqrt{-37(5+2\sqrt{5})}  \end{pmatrix}.
\end{cases}
\end{align*}

Putting
$
\lambda_1=u(\alpha_1),  \lambda_2=u(\alpha_2),  \lambda_3= 2 u(\alpha_3)-u(\alpha_4)$ and 
$\lambda_4=u(\alpha_4)-u(\alpha_3) +u(\alpha_1), $
we have
 \begin{align*}
 E(\lambda_j,\lambda_k)=
 \begin{pmatrix}
 0 & 0 & 1 & 0 \\
 0 & 0 &  0 & 1 \\
 -1 & 0 & 0 &0 \\
 0 & -1 & 0 &0
  \end{pmatrix}.
 \end{align*}
We set
 $
(M_1M_2)=(\lambda_1\lambda_2\lambda_3\lambda_4),
$
where
$M_1,M_2 \in M(2,\mathbb{C}).$
So, we have
$
\Omega=-M_2^{-1} M_1 \in \mathfrak{S}_2.
$

Set
$
\begin{pmatrix} A _0& B_0 \\ C_0 &D_0  \end{pmatrix} =
\begin{pmatrix}
0 & -1 & 0 & 1\\
1 &0 &1 & 0\\
0 & -1 &0 & 0\\
-1 & 0  & 0 &0
  \end{pmatrix}
  \in Sp(2,\mathbb{Z}).
$
 Then, we have
$
 \tilde{\Omega}=(A_0 \Omega + B_0)(C_0 \Omega + D_0)^{-1}
=
\begin{pmatrix}
\tau_1 & \tau_2 \\
\tau_2 & \tau_3
\end{pmatrix},
$
where
$
\tau_1= \sqrt{-1}\sqrt{\frac{37}{10}(5+\sqrt{5})},
\tau_2=\frac{3}{2} -\frac{1}{2}\sqrt{-1}\sqrt{185-\frac{74}{\sqrt{5}}} $ and 
$
\tau_3= -\frac{3}{2}+\frac{1}{2}\sqrt{-1}\sqrt{481+\frac{74}{\sqrt{5}}}. 
$
 This satisfies
$
\tau_1  - \tau_2 -\tau_3=0.
$
So, from $\tau_2$ and $\tau_3$,
we have $(z_1^0,z_2^0)\in\mathbb{H}\times \mathbb{H}$ as in (\ref{z10,z20}).
According to Theorem  \ref{ThmCF}, 
using the pair of modular functions $(X,Y)$ of (\ref{XY}),
the unramified class field $k_0$ over $K=\mathbb{Q}\left(\sqrt{-37(5+2\sqrt{5})}\right)$ corresponding to the group $H_0$ of (\ref{H0}) is given by
$
k_0=K(X(z_1^0,z_2^0),Y(z_1^0,z_2^0)).
$

Due to \cite{HHRWH},
the class number $h_K$ of $K=\mathbb{Q}\left(\sqrt{-37(5+2\sqrt{5})}\right)$ is equal to $10$.
Therefore, 
$I_K/P_K$ is isomorphic to $(\mathbb{Z}/2\mathbb{Z}) \oplus (\mathbb{Z}/5\mathbb{Z}).$
Applying Theorem \ref{ThmHantei} to this case,
we have
$[k_0:K]=5.$
Namely, 
this case   gives an example of an unramified class field $k_0$ corresponding to $H_0$ of (\ref{H0})
which gives a non-trivial extension $k_0/K$.

\section*{Acknowledgment}
The author would like to thank Professor Hironori Shiga for helpful advises and  valuable suggestions,
and also to Professor Kimio Ueno  for kind encouragements.
This work is supported by 
The JSPS Program for Advancing Strategic International Networks to Accelerate the Circulation of Talented Researchers
"Mathematical Science of Symmetry, Topology and Moduli, Evolution of International Research Network based on OCAMI",
The Sumitomo Foundation Grant for Basic Science Research Project (No.150108)
and
Waseda University Grant for Special Research Project (2014B-169 and 2015B-191).

\begin{center}
\hspace{8.8cm}\textit{Atsuhira  Nagano}\\
\hspace{8.8cm}\textit{Department of Mathematics}\\
\hspace{8.8cm}\textit{King's College London}\\
\hspace{8.8cm}\textit{Strand, London, WC2R 2LS}\\
\hspace{8.8cm}\textit{The United Kingdom}\\
 \hspace{8.8cm}\textit{(E-mail: atsuhira.nagano@gmail.com)}
  \end{center}

\end{document}